\documentclass[12pt]{amsart} 
\usepackage{amssymb,amscd,amsmath}

\usepackage{graphicx,float} 
\usepackage{color}
\usepackage[all,ps,cmtip]{xy} 

\usepackage[norelsize]{algorithm2e}

 \usepackage[colorlinks=true,citecolor=blue, urlcolor=blue, breaklinks, linkcolor=black]{hyperref}

  \usepackage{mathrsfs}
 \let\mathcal\mathscr

 \DeclareRobustCommand{\SkipTocEntry}[5]{}

\def\bw#1#2{\textstyle{\bigwedge\hskip-0.9mm^{#1}}\hskip0.2mm{#2}}

\def\cha{characteristic}

\def\av{abelian variety}
\def\avs{abelian varieties}

\def\dra{\dashrightarrow}

\def\isom{\simeq}
\def\eps{\varepsilon}
\def\ie{\hbox{i.e.}}

\DeclareMathOperator{\isomto}{\stackrel{{}_{\scriptstyle\sim}}{\to}}

\DeclareMathOperator{\Card}{Card}

\DeclareMathOperator{\Br}{Br}
\DeclareMathOperator{\Disc}{Disc}

\DeclareMathOperator{\Gal}{Gal}
\DeclareMathOperator{\Gr}{Gr}
\DeclareMathOperator{\Tr}{Tr}

\DeclareMathOperator{\Id}{Id}

\DeclareMathOperator{\NS}{NS}

\DeclareMathOperator{\PicVar}{PicVar}
\DeclareMathOperator{\pr}{pr}

\DeclareMathOperator{\Sing}{Sing}

\DeclareMathOperator{\Sym}{Sym}

\def\llra{\hbox to 10mm{\rightarrowfill}}
\def\lllra{\hbox to 15mm{\rightarrowfill}}

\def\dra{\dashrightarrow}

\def\Z{{\bf Z}}

\def\P{{\bf P}}
\def\R{{\bf R}}
\def\Q{{\bf Q}}
\def\C{{\bf C}}

\def\F{{\bf F}}

\def\T{{\bf T}}

\def\eps{{\varepsilon}}

\def\kk{{\bf{k}}}

\def\mr{{\mathbf r}}

\def\phi{\varphi}

\def\gm{{\mathfrak m}}
\def\gp{{\mathfrak p}}
\def\gU{{\mathfrak U}}
\def\gS{{\mathfrak S}}
\def\gF{{\mathfrak F}}

\def\cE{{\mathcal E}}

\def\cO{{\mathcal O}}

\def\cS{{\mathcal S}}

\def\cX{{\mathcal X}}

\def\LL{{\mathbb L}}

\newtheorem{theo}{Theorem}[section]
\newtheorem{prop}[theo]{Proposition}
\newtheorem{lemm}[theo]{Lemma}
\newtheorem{coro}[theo]{Corollary}

\newtheorem{rema}[theo]{Remark}

\newtheorem{exam}[theo]{Example}

\begin{document}

\title{Lines on cubic hypersurfaces over finite fields}

\author[O.~Debarre]{Olivier Debarre}
\address{D\'epartement Math\'ematiques et Applications\\UMR CNRS 8553\\PSL Research University\\\'Ecole Normale Su\-p\'e\-rieu\-re\\45 rue d'Ulm, 75230 Paris cedex 05, France}
\email{{olivier.debarre@ens.fr}}
\urladdr{\url{http://www.math.ens.fr/~debarre/}}

\author[A.~Laface]{Antonio Laface}
\address{
Departamento de Matem\'atica,
Universidad de Concepci\'on,
Casilla 160-C,
Concepci\'on, Chile}
\email{alaface@udec.cl}

\author{Xavier Roulleau}
\address{Laboratoire de Math\'ematiques et Applications, Universit\'e de Poitiers, UMR CNRS 7348, T\'el\'eport 2 - BP 30179 - 86962 Futuroscope Chasseneuil, France
}
\email{xavier.roulleau@math.univ-poitiers.fr}
\urladdr{\url{http://www-math.sp2mi.univ-poitiers.fr/~roulleau/}}


 
\thanks{
The second author was partially supported 
by Proyecto FONDECYT Regular N.\ 1150732
and Proyecto Anillo ACT 1415 PIA Conicyt.
}

 \begin{abstract}
We show that any smooth cubic hypersurface of dimension $n $ defined over a finite field $\F_q$ contains a line defined over   $\F_q$  in each of the following cases:
\begin{itemize}
\item $n=3$ and $q\ge 11$;
\item $n=4$ and $q\ne3$;
\item $n\ge 5$.
\end{itemize}
For a smooth cubic threefold $X$, the variety of lines contained in $X$ is a smooth projective surface $F(X)$ for which the Tate conjecture holds, and  we obtain information about the Picard number of $F(X)$ and the 5-dimensional principally polarized Albanese variety $A(F(X))$. 
 \end{abstract}

 \maketitle

\section{Introduction}

The study of rational points on hypersurfaces in the projective space defined over a finite field has a long history.\ Moreover, if $X\subset \P^{n+1}$ is a (smooth) cubic hypersurface, the (smooth) variety $F(X)$ parametrizing lines contained in $X$ is an essential tool for the   study of the geometry of $X$.\ Therefore, it seems natural to investigate $F(X)$ when $X$ is a cubic hypersurface    defined over a finite field $\F_q$ and the first question to ask is whether $X$ contains a line defined over $\F_q$.

One easily finds smooth cubic surfaces defined over $\F_q$ containing no $\F_q$-lines, with $q$ arbitrarily large.\ On the other hand, if $\dim(X)\ge 5$, the variety $F(X)$, when smooth, has ample anticanonical bundle and it follows from powerful theorems of Esnault and Fakhruddin--Rajan that $X$ always contains an $\F_q$-line (Section  \ref{5ormore}).\ So the interesting cases are when $\dim(X)=3$ or~$4$.

When $X$ is a smooth cubic threefold, $F(X)$ is a smooth surface of general type.\ Using a  formula of Galkin--Shinder which relates the number of $\F_q$-points on $F(X)$ with the number of $\F_q$- and $\F_{q^2}$-points on $X$   (Section \ref{segs}), we  find the zeta function of $F(X)$   (Theorem~\ref{main}).\ Using the Weil conjectures, we obtain that a smooth $X$ always contains $\F_q$-lines when $q\ge 11$ (Theorem~\ref{prop6}).\ Using a computer, we produce examples of smooth cubic threefolds containing no $\F_q$-lines for $q\in\{2,3,4,5\}$ (Section \ref{noli}), leaving only  the cases where $q\in\{7,8,9\}$ open, at least when $X$ is smooth.

Theorem \ref{main} can also be used for explicit computations of  the zeta function of $F(X)$.\ 
  For that, one needs to know the number of $\F_{q^r}$-points of $X$ for sufficiently many $r$.\ Direct computations are possible   for small $q$   or when $X$ has   symmetries (see Section \ref{exfermat} for   Fermat hypersurfaces, Section \ref{kleins} for the Klein threefold, and \cite{ked} for cyclic cubic threefolds).\ If $X$ contains an $\F_q$-line, it is in general faster to use the structure of conic bundle on $X$ induced by projection from this line, a method initiated by   Bombieri and Swinnerton-Dyer  in 1967 (Section~\ref{secbsd}).\ This is illustrated by an example in Section \ref{algos}, where we compute the zeta function of a cubic $X$ and of its Fano surface $F(X)$ in characteristics up to  $31$.\ In all these examples, once one knows the zeta function of $F(X)$,   the Tate conjecture (known in this case; see Remark \ref{tatec}) gives its Picard number.\ It is also easy to determine whether    its 5-dimensional Albanese variety $A(F(X))$ is simple, ordinary, supersingular...

Singular cubics tend to contain more lines (Example \ref{cex}).\ When $X$  is a cubic threefold with a single node, the geometry of $F(X)$ is closely related to that of a smooth genus-4 curve (\cite{cg}, \cite{kvdg}; see also \cite[Example 5.8]{gash}).\ Using the results of \cite{HLT} on pointless curves of genus~$4$, we prove that $X$ always contains $\F_q$-lines when $q\ge 4$ (Corollary \ref{coronodal}) and produce examples for $q\in\{2,3\}$ where $X$   contains no $\F_q$-lines
(Section \ref{nodal}).

When $X$ is a smooth cubic fourfold, $F(X)$ is a smooth fourfold with trivial canonical class.\ Using again the Galkin--Shinder formula, we compute the 
zeta function of $F(X)$  (Theorem \ref{main4})  
 and deduce from the Weil conjectures that $X$ contains an $\F_q$-line when $q\ge 5$ (Theorem \ref{th52}).\ 
 A trick using a theorem of Mazur based on crystalline cohomology (provided by K.~Kedlaya)   proves  that $X$ also contains an $\F_q$-line when $q$ is any power of $2$ (Proposition~\ref{ked}).\
 The case  $q=3$ is the only one left open;  calculations on a computer suggest that any cubic fourfold defined over $\F_3$ should contain an $\F_3$-line.

\subsection*{Acknowledgements}
 This collaboration   started  after a mini-course given in 2015 by the first author for the  CIMPA School  ``II Latin American School of Algebraic Geometry and Applications'' given in Cabo Frio, Brazil.\ We thank CIMPA for financial support and the organizers C.~Araujo and S.~Druel for making this  event successful.\
 Many thanks also to N.~Addington (who corrected an error in the original equation in Section~\ref{sec542}), 
 D.~Bragg (who corrected an error in  the original statement of Theorem~4.12), 
 F.~Charles, S.~Elsenhans, B.~van Geemen, F.~Han, E.~Howe, T.~Katsura, Ch.~Liedtke, and O.~Wittenberg for useful correspondences and conversations.\ Special thanks go to K.~Kedlaya, who spotted an error in the calculations in our original  proof of Theorem~\ref{th52} and explained how to fix it to  obtain an improved result; he also provided us with a proof of the existence of lines on smooth cubic fourfolds defined over $\F_4$ (Proposition~\ref{ked}) and kindly allowed us to include it in the present version of this article.\
 The computations made for this paper were done with the computer algebra programs {\tt Magma}~(\cite{bc}) and {\tt Sage}.

\section{Definitions and tools}

\subsection{The  Weil and Tate conjectures}\label{sedwc}

Let $\F_q$ be a finite field with $q$ elements and let $\ell$ be a prime number prime to $q$.

Let $Y$ be a   projective variety of dimension $n$ defined over $\F_q$.\ For every integer  $r\ge 1$,   set
$$N_r(Y):=\Card\bigl(Y(\F_{q^r})\bigr) 
$$
and define the   {\em zeta function}
$$Z(Y,T):=\exp\Bigl(\sum_{r\ge 1}N_r(Y)\frac{T^r}{r}\Bigr).
$$

Let $\overline{\F_q}$ be an algebraic closure of $\F_q$  and let $\overline{Y}$ be the  variety obtained from $Y$ by extension of scalars from $\F_q$ to $\overline{\F_q}$.\ The Frobenius morphism $F\colon \overline{Y}\to \overline{Y}$  acts  on the \'etale cohomology  $H^\bullet(\overline{Y} ,\Q_\ell)$ by a $\Q_\ell$-linear map which we denote by $F^*$.\  Grothendieck's Lefschetz Trace formula (\cite[Theorem 13.4, p.~292]{mil}) states that, for all integers $r\ge 1$, one has
\begin{equation}\label{grolef}
N_r(Y)=\sum_{0\le i\le 2n}(-1)^i\Tr\bigl(F^{*r},H^i(\overline{Y} ,\Q_\ell)\bigr).
\end{equation}
If $Y$ is moreover smooth, the   Weil conjectures proved by Deligne in \cite[Th\'eor\`eme (1.6)]{del} say that for each $i$, the (monic) characteristic polynomial 
$$Q_i(Y,T):=\det\bigl(T\Id-F^*,H^i(\overline{Y} ,\Q_\ell)\bigr)$$
has integral coefficients and is independent of $\ell$ (in particular, so is its degree  $b_i(Y):=h^i(\overline{Y} ,\Q_\ell)$, called the {\em $i$-th Betti number} of $Y$) and all the conjugates of its complex roots $\omega_{ij}$ have modulus~$q^{i/2}$.\ Poincar\'e duality implies $b_{2n-i}(Y)=b_i(Y)$ and $\omega_{2n-i,j}=q^n/\omega_{ij}$ for all $1\le j\le b_i(Y)$.

We can rewrite the trace formula \eqref{grolef} as
\begin{equation}\label{grolef2}
N_r(Y)=\sum_{0\le i\le 2n}(-1)^i\sum_{j=1}^{b_i(Y)}\omega_{ij}^r 
\end{equation}
or
\begin{equation}\label{zeta}
Z(Y,T)=\prod_{0\le i\le 2n} P_i(Y,T)^{(-1)^{i+1}}.
\end{equation}
Finally, it is customary   to introduce the polynomials
\begin{equation}\label{pi}
P_i(Y,T):=\det\bigl(\Id-TF^*,H^i(\overline{Y} ,\Q_\ell)\bigr)=T^{b_i(Y)}Q_i\Bigl(Y,\frac{1}{T}\Bigr)=\prod_{j=1}^{b_i(Y)}(1- \omega_{ij}T).
\end{equation}

{\em Whenever $i$ is odd,}  the real roots of $Q_i( Y,T)$ have even multiplicities (\cite[Theorem~1.1.(b)]{elja}), hence $b_i(Y)$ is even.
We can therefore assume   $\omega_{i,j+b_i(Y)/2}=\bar\omega_{ij}$ for all $1\le j\le  b_i(Y)/2$, or 
$T^{b_i(Y)}Q_i( Y,q^i/T)=q^{ib_i(Y)/2}Q_i( Y,T)$.\ If $m:={b_1(Y)/2}$, we will write
\begin{equation}\label{rec}
 Q_1(Y,T)=T^{2m}+a_1T^{2m-1}+\dots+ a_{m-1}T^{m+1}+a_mT^m+qa_{m-1}T^{m-1}+\dots+q^{m-1}a_1T+q^m.
\end{equation}

The Tate conjecture for divisors on $Y$ states that the $\Q_\ell$-vector space in $H^2\bigl(\overline{Y} ,\Q_\ell(1)\bigr)$ generated by  classes of $\F_q$-divisors   is equal to the space of $\Gal(\overline{\F_q}/\F_q)$-invariants classes and that its dimension is equal to the multiplicity of $q$ as a root of the polynomial $Q_2(Y,T)$ (\cite[Conjecture~2, p.~104]{tate2}).

\subsection{The Katz trace formula}\label{skatz}

Let $Y$ be   a proper scheme of dimension $n$ over $\F_q$.\ The endomorphism $f\mapsto f^q$ of  $\cO_Y$ induces an $\F_q$-linear endomorphism $\gF_q$  of the $\F_q$-vector space $H^\bullet(Y,\cO_Y)$ and for all $r\ge 1$, one has   (\cite{kat}, Corollaire 3.2) 
\begin{equation}\label{trka}
N_r(Y)\cdot 1_{\F_q}= \sum_{j=0}^n (-1)^j \Tr \bigl(\gF_q^r, H^j(Y,\cO_Y)\bigr)\qquad \hbox{in }\F_q. 
\end{equation}
 In particular,   the right side, which is a priori in $\F_q$, is actually in the prime subfield of $ \F_q$.
 
\subsection{The Galkin--Shinder formulas}\label{segs}

Let $X\subset \P^{n+1}_{\F_q}$ be a reduced cubic hypersurface   defined over  $\F_q$, with singular set   $\Sing(X)$. 

We let $F(X)\subset \Gr(1, \P^{n+1}_{\F_q})$ be the scheme of lines contained in $X$; it is also defined over~$\F_q$.\ When $n\ge 3$ and     $\Sing(X)$ is finite, $F(X)$ is a  local complete intersection of dimension $2n-4$, smooth if $X$ is smooth, and geometrically connected   (\cite[Theorem (1.3)   and Corollary (1.12)]{ak}).

In the Grothendieck ring of varieties over $\F_q$, one has   the   relation (\cite[Theorem 5.1]{gash})
\begin{equation}\label{gsfgrot}
\LL^2[F(X)]=[X^{(2)}]-(1 + \LL^n)[X]  +\LL^n[\Sing(X )], 
\end{equation}
where $X^{(2)}:=X^2/\gS_2$ is the symmetric square of $X$ and, as usual, $\LL$ denotes the class of the affine line.\ Together with the relation \cite[(2.5)]{gash}, it implies that, for all $r\ge 1$, we have (\cite[Corollary~5.2.3)]{gash})
\begin{equation}\label{gsf}
N_r\bigl(F(X)\bigr)=\frac{N_r(X)^2-2(1+q^{nr})N_r(X)+N_{2r}(X)}{2q^{2r}}+q^{ (n-2)r}N_r\bigl(\Sing(X)\bigr).
\end{equation}

\subsection{Abelian varieties over finite fields}\label{sso}

Let $A$ be an \av\  of dimension $n$ defined over a finite field $\F_q$ of \cha\ $p$ and let $\ell$ be a prime number prime to $p$.\ The $\Z_\ell$-module $H^1(\overline A,\Z_\ell)$ is free of rank $2n$ and there is an isomorphism 
\begin{equation}\label{wed}
\bw{\bullet}{H^1(\overline A,\Q_\ell)} \isomto H^\bullet(\overline A,\Q_\ell)
\end{equation}
of $\Gal(\overline{\F_q}/\F_q)$-modules.

An elliptic curve $E$ defined over $\overline{\F_q}$ is {\em supersingular} if its only $p$-torsion  point   is $0$.\ All
supersingular elliptic curves are isogenous.
The 
  \av\ $A$   is {\em supersingular} if $A_{\overline{\F_q}}$ is isogenous  to $E^n$, where $E$ is a supersingular elliptic curve (in particular,   any two supersingular \avs\ are isogenous over $\overline{\F_q}$).
The  following  conditions are equivalent (\cite[Theorems~110, 111, and 112]{Hu})
\begin{itemize}
 \item[(i)]  $A$  is supersingular;
\item[(ii)]  $Q_1(A_{\F_{q^r}},T)=(T\pm q^{r/2} )^{2n}$ for some $r\ge 1$; 
\item[(iii)]  $\Card\bigl(A(\F_{q^r} )\bigr)=(q^{r/2}\pm 1 )^{2n}$ for some $r\ge 1$;
\item[(iv)] each complex root of $Q_1(A,T)$ is $\sqrt{q}$ times a root of unity;
\item[(v)] in the notation of \eqref{rec}, if $q=p^r$, one has $p^{\lceil rj/2\rceil}\mid a_j$ for all $j\in\{1,\dots,n\}$.
\end{itemize}
If condition (ii) is satisfied, one has  $Q_2(A_{\F_{q^r}},T)=(T-q^r )^{n(2n-1)}$ and the Tate conjecture, which holds for divisors on abelian varieties, implies that the Picard number of $A_{\F_{q^r}}$, hence also the geometric Picard number of  $A$, is $n(2n-1)$, the maximal possible value.\ Conversely, when  $n>1$, if $A_{\F_{q^r}}$ has maximal Picard number for some $r$, the \av\ $A
$ is supersingular.

The abelian variety $A$ is {\em ordinary}  if it contains $p^n$ (the maximal possible number) $p$-torsion $\overline{\F_q}$-points.\ This is equivalent to the coefficient $a_n$ of $T^n$ in  $Q_1(A,T)$ being prime to $p$;
if this is the case,  $A$   is simple (over $\F_q$)  if and only if  the polynomial $Q_1(A,T)$ is  irreducible (see \cite[Section~2]{HZ}). 

\section{Cubic surfaces}

There exist smooth cubic surfaces defined over $\F_q$ containing no $\F_q$-lines, with $q$ arbitrarily large.\ This is the case for example for the diagonal cubics defined by
$$x_1^3+x_2^3+x_3^3+ax_4^3=0,$$
where $a\in \F_q$ is not a cube.\ If  $q\equiv 1\pmod3$, there is such an $a$, since  there are   elements of order~$3$ in $\F_q^\times$, hence the morphism $\F_q^\times\to \F_q^\times$, $x\mapsto x^3$ is not injective, hence not surjective.

\section{Cubic threefolds}\label{secuth}

\subsection{The zeta function of the surface of lines}

Let    $X\subset \P^4_{\F_q}$ be a {\em smooth} cubic hypersurface defined over $\F_q$.\ Its Betti numbers   
are $1$, $0$, $1$, $10$, $1$, $0$, $1$, and the eigenvalues of the Frobenius morphism acting on the 10-dimensional vector space $H^3(\overline X,\Q_\ell)$  
 are all divisible by $q$ as algebraic integers (\cite[Remark~5.1]{kat}).\ We can therefore write \eqref{grolef} as
$$N_r(X)=1+q^r+q^{2r}+q^{3r}-q^r\sum_{j=1}^{10}\omega_j^r,$$
where, by the Weil conjectures proved by Deligne (Section \ref{sedwc}), the complex algebraic integers~$\omega_j$ (and all their  conjugates) have modulus 
$\sqrt{q}$.\ The trace formula \eqref{zeta} reads
$$Z(X,T)=\frac{P_3(X,T)}{(1-T)(1-qT)(1-q^2T)(1-q^3T)},
$$
where $P_3(X,T) =\prod_{j=1}^{10}(1-q\omega_jT)$.\ If we set 
\begin{equation}\label{mr}
M_r(X):=
\frac1{q^r}\bigl(
N_r(X)-(1+q^r+q^{2r}+q^{3r})\bigr) 
=-\sum_{j=1}^{10}\omega_j^r,
\end{equation}
we obtain
\begin{equation}\label{p3}
P_3(X,T)=\exp \Bigl (\sum_{r\ge 1}M_r(X)\frac{(qT)^r}{r}\Bigr)
.
\end{equation}

We will show in Section \ref{secbsd} that the numbers $M_r(X)$ have  geometric significance.  

\begin{theo}\label{main}
Let    $X\subset \P^4_{\F_q}$ be a {\em smooth} cubic hypersurface defined over $\F_q$ and let $F(X)$ be the smooth surface of lines contained in $X$.\ With the notation \eqref{pi}, we have 
\begin{eqnarray*}
P_1(F(X),T)&=& P_3(X,T/q)=:\prod_{1\le j\le 10}(1-\omega_jT),\\
P_2(F(X),T)&=& \prod_{1\le j<k\le 10}(1-\omega_j\omega_kT),\\
P_3(F(X),T)&=& P_3(X,T )=\prod_{1\le j\le 10}(1-q\omega_jT),
\end{eqnarray*}
where the complex numbers $\omega_1,\dots,\omega_{10}$ have  modulus $\sqrt{q}$.\ In particular,  
\begin{equation}\label{zetafx}
Z(F(X),T)=\frac{\prod_{1\le j\le 10}(1-\omega_jT)\prod_{1\le j\le 10}(1-q\omega_jT)}{(1-T)(1-q^2T)\prod_{1\le j<k\le 10}(1-\omega_j\omega_kT)}
.\end{equation}
\end{theo}

\begin{proof} There are several ways to prove this statement.\ The first is to prove that there are isomorphisms
\begin{equation*} 
 H^3(\overline X,\Q_\ell)\isomto H^1\bigl(\overline{F(X)},\Q_\ell(-1)\bigr)\quad{\rm and}\quad  \bw{2}{H^1(\overline{F(X)},\Q_\ell)}\isomto H^2(\overline{F(X)},\Q_\ell)
\end{equation*}
of $\Gal(\overline{\F_q}/\F_q)$-modules.\ The first isomorphism  holds with $\Z_\ell$-coefficients: if we introduce the incidence variety $I=\{(L,x)\in F(X)\times X\mid x\in L\}$ with its projections $\pr_1\colon I\to F(X)$ and $\pr_2\colon I\to X$, it is given by $\pr_{1*}\pr_2^*$ (\cite[p.~256]{chpi}).\  
 The second isomorphism follows, by standard arguments using smooth and proper base change, from the analogous statement in singular cohomology, over $\C$, which is proven in 
   \cite[Proposition 4]{rou1}. 
   
 These isomorphisms (and Poincar\'e duality)  then imply the formulas for the polynomials $P_i(F(X),T)$ given in the theorem.
   
   Alternatively, simply substituting in the definition of $Z(F(X),T)$ the values for  $N_r(F(X))$ given by the 
   Galkin--Shinder formula \eqref{gsf} directly gives \eqref{zetafx}, from which one deduces the formulas for the polynomials $P_i(F(X),T)$.\end{proof}

\begin{rema}[The Tate conjecture for $F(X)$]\upshape\label{tatec}
The Tate conjecture for the surface $F(X)$ (see Section \ref{sedwc}) was proved in \cite{rou1} over any field $\kk$ of finite type over the prime field, {\em of \cha\ other than 2.} This last restriction can in fact be lifted as follows: the proof in \cite{rou1} rests  on  the following two facts 
\begin{itemize}
\item[a)]  $F(X)$ maps to  its (5-dimensional) Albanese variety $A(F(X))$ onto a surface with class a multiple of $\theta^3$, where $\theta$ is a principal polarization on $A(F(X))$;
\item[b)] $b_2(A(F(X)))=b_2(F(X))$.
\end{itemize}
Item a) is proved (in \cha\ $\ne 2$) 
  via the theory of Prym varieties (\cite[Proposition~7]{bea2}).
  For item b), we have $\dim(A(F(X)))=h^1(F(X),\cO_{F(X)})=5$ (\cite[Proposition (1.15)]{ak}), hence $b_2(A (F(X)))=\binom{2\dim(A(X))}{2}=45$, whereas $b_2(F(X))=\deg\bigl(P_2(F(X),T)\bigr)=45$ by Theorem~\ref{main}.
  
  To extend a) to all characteristics, we consider $X$ as the reduction modulo the maximal ideal~$\gm$ of a smooth cubic $\cX$ defined over a valuation ring of \cha\ zero.\ There is a ``difference morphism'' $\delta_{F(X)}\colon F(X)\times F(X)\to A(F(X))$, defined over $\kk$, which is the reduction modulo $\gm$ of the analogous morphism  $\delta_{F(\cX)}\colon F(\cX)\times F(\cX)\to A\bigl(F(\cX)\bigr)$.\ By  \cite[Proposition~5]{bea2}, the image of $\delta_{F(\cX)}$ is a divisor which defines a principal polarization $\vartheta$ on $A\bigl(F(\cX)\bigr)$, hence the image of $\delta_{F(X)}$ is also a principal polarization  on $A(F(X))$, defined over $\kk$. 
  
  Since the validity of the Tate conjecture is not affected by passing to a finite extension of $\kk$, we may assume that $F(X)$ has a $\kk$-point, which we lift to  $F(\cX)$.\ We can then define Albanese morphisms, and   $a_{F(X)}\colon F(X)\to A(F(X))$ is the reduction modulo $\gm$ of $a_{F(\cX)}\colon F(\cX)\to A(F(\cX))$.\  The image of $a_{F(\cX)}$ has class 
 $\vartheta^3/3!$ (\cite[Proposition 7]{bea2}), hence the image of $a_{F(X)}$ also has class  $(\vartheta\vert_{A(X)})^3/3!$ (this class is not divisible in $H^6(A(X),\Z_\ell)$, hence   $a_{F(X)}$ is generically injective).\ This proves a), hence the Tate conjecture for $F(X)$, in all characteristics.
 
Going back to the case where $\kk$ is finite, Theorem \ref{main} implies the equality $Q_2(F(X),T)= Q_2(A(F(X)),T)$.\ Since the Tate conjecture holds for divisors on abelian varieties, this proves that   $F(X)$ and $A(F(X))$ have the same Picard numbers, whose maximal possible value is 45.
\end{rema}
 
 \begin{coro}\label{pic}
Let $2m_\pm$ be the multiplicity of the root $\pm\sqrt{q}$ of $Q_1(F(X),T)$ and let $m_1,\dots, m_c$ be the multiplicities of the pairs of nonreal conjugate roots of $Q_1(F(X),T)$, so that $m_++m_-+ \sum_{i=1}^cm_i=5$.\ The Picard number of $F(X)$ is then
$$\rho(F(X))=m_+(2m_+-1)+m_-(2m_--1)+\sum_{i=1}^c m_i^2 
.$$
We have $\rho(F(X))\ge 5$, with equality if and only if $Q_1(F(X),T)$ has no multiple roots.

If $q$ is not a square, the  possible Picard numbers are all odd numbers between $5$ and $13$, and $17$ and $25$.

If $q$ is a square, the possible Picard numbers are all odd numbers between $5$ and $21$, and $25$, $29$, and $45$.\ We have  $\rho(F(X))=45 $ if and only if $Q_1(F(X),T)= (T\pm \sqrt q )^{10}$.
  \end{coro}

\begin{proof}
The Tate conjecture holds for divisors on $F(X)$ (Remark \ref{tatec}).\ As explained at the end of Section \ref{sedwc}, it says that the rank of the Picard group is the multiplicity of $q$ as a root  of $Q_2(F(X),T)$.\ The remaining statements then follow  from   Theorem \ref{main}
by inspection of all possible cases for the values of 
$m_+,m_-,m_1,\dots, m_c$.
\end{proof}

\begin{rema}[The Artin--Tate conjecture for $F(X)$]\upshape\label{atcsrem}
Let    $X\subset \P^4_{\F_q}$ be a  smooth  cubic hypersurface defined over $\F_q$.\ By Remark \ref{tatec}, the Tate conjecture holds for
the (smooth projective) surface $F(X)$. 

We denote by $\NS(F(X))$ its  N\'eron--Severi group, by $\NS(F(X))_{\rm tor}$
its torsion subgroup, and by $\rho(F(X))$ and $\Disc\bigl(\NS(F(X))\bigr)$
   the rank and the discriminant
of the lattice $\NS(Y)/\NS(Y)_{\rm tor}$, respectively.\
Finally, we let 
 $\PicVar (F(X))$ be its Picard variety, which  is reduced of dimension~$5$
 \cite[Lem.~1.1]{Tyurin1}. 
 
 In that situation, the main result of \cite{mil2}    (the Artin--Tate
Conjecture; see \cite[Section~4.4 (C)]{milt}) is that when $q$ is odd,\footnote{When $q$ is even, Tate still proved in \cite[Theorem 5.2]{tatedix} that the equality \eqref{atc} holds modulo a power of $2$.} the Brauer group $\Br\bigl(F(X)\bigr)$ is finite\footnote{Its order is then a square by \cite{llr}.} and 
\begin{equation}\label{atc}
D_q\bigl(F(X)\bigr):=\lim_{s\to1}\frac{P_{2}(F(X),q^{-s})}{(1-q^{1-s})^{\rho(F(X))}}=\frac{
\Card\bigl(\Br(F(X))\bigr)|\Disc\bigl(\NS(F(X))\big)|}{q^{\alpha }\Card\bigl(\NS(F(X))_{\rm tor}\bigr)^{2}},
\end{equation} 
 where $\alpha :=\chi(F(X),\cO_{F(X)})-1+\dim\bigl(\PicVar(F(X))\bigr)=10$.
 
 In all the  examples where we   computed the zeta function of $F(X)$, the equality \eqref{atc} will imply that the quotient $\frac{
\Card (\Br (F(X) ) )|\Disc (\NS (F(X) ) )|}{ \Card (\NS (F(X) )_{\rm tor} )^{2}}$ is an integer (see \eqref{atc1}, \eqref{atc2}, \eqref{atc3}, \eqref{atc4}, \eqref{atc5}, \eqref{atc44}), suggesting that $\NS(F(X))$ might be torsion-free.
\end{rema}

 In \cha\ $0$, the N\'eron--Severi group $\NS(F(X))$ is indeed torsion-free: this follows from the fact that $H^2(F(X),\Z)$ is a free abelian group (\cite[Lemmas 9.3 and 9.4]{cg}).\ Standard arguments using smooth and proper base change then
imply that in \cha\ $p>0$, the N\'eron--Severi group has no prime-to-$p$ torsion.

\subsection{Existence of lines on smooth cubic threefolds over large finite fields}

We can now bound the number of $\F_q$-lines on a smooth cubic threefold defined over $\F_q$.

\begin{theo}\label{prop6}
Let $X$ be a smooth cubic threefold defined over $\F_q$ and let $N_1(F(X))$ be the   number of $\F_q$-lines contained in $X$.\ We have
$$N_1(F(X))\ge
\begin{cases}
1+45q+q^2-10(q+1)\sqrt q& {\it if }\  q\ge 64;\\
1+13q+q^2-6(q+1)\sqrt q& {\it if }\ 16\le q\le 61;\\
1-3q+q^2-2(q+1)\sqrt q& {\it if } \  q\le 13.
\end{cases}
$$
In particular,  
  $X$ contains at least 10 $\F_q$-lines if $q\ge 11$.
  
  Moreover, for all $q$,
  $$N_1(F(X))\le 1+45q+q^2+10(q+1)\sqrt q
  .$$
\end{theo}

\begin{proof}
As we saw in Section  \ref{sedwc}, we can write the roots of  $Q_1(F(X),T)$  
as $\omega_1,\dots ,\omega_5 ,\overline\omega_1,\dots ,\overline\omega_5$.\ The $r_j:=\omega_j+\overline \omega_j$ are then real numbers in $[-2\sqrt{q},2\sqrt{q}]$ and, by \eqref{grolef2} and Theorem \ref{main}, we  have
\begin{eqnarray*}
N_1\bigl( F(X)\bigr)&=&1-\sum_{1\le j\le 5}r_j+5q+\sum_{1\le j<k\le 5}(\omega_j  \omega_k+ \overline\omega_j  \omega_k+ \omega_j  \overline\omega_k+ \overline\omega_j  \overline\omega_k)-\sum_{1\le j\le 5}qr_j+q^2\\
&=&1+5q+q^2-(q+1)\sum_{1\le j\le 5}r_j+\sum_{1\le j<k\le 5}r_j  r_k\\
&=:&F_q( r_1,\dots,r_5 ).
\end{eqnarray*}
Since the   real function
$F_q\colon [-2\sqrt{q},2\sqrt{q}]^5\to \R$ is   {\em linear in each variable,} its extrema are reached on the boundary of its domain, \ie,  at one of the points $2\sqrt q\,(\pm 1,\dots,\pm 1)$.\ At such a point $\mr_l$ (with $l$ positive coordinates), we have  
$$F_q(\mr_l)=1+5q+q^2-2(2l-5)(q+1)\sqrt q+  \frac12\bigl( 4q(2l-5)^2 -20q\bigr).
$$
The  minimum
is  obviously reached      for $l\in\{3,4,5\}$, the maximum for $l= 0$, and the rest is easy. 
 \end{proof}

\subsection{Computing techniques: the Bombieri--Swinnerton-Dyer method}\label{secbsd}

By Theorem~\ref{main}, the zeta function of the surface $F(X)$  of lines contained in a smooth cubic threefold $X\subset \P^4_{\F_q}$ defined over $\F_q$  is completely determined by the roots $q\omega_1,\dots,q\omega_{10}$ of the degree-10 \cha\ polynomial 
of the Frobenius morphism acting on $H^3(\overline X,\Q_\ell)$.\ If one knows the numbers of points of $X$ over   sufficiently many finite extensions of $\F_q$, these roots can be computed from the relations
$$\exp \Bigl (\sum_{r\ge 1}M_r(X)\frac{T^r}{r}\Bigr)=P_3(X,T/q)=P_1(F(X),T)=\prod_{1\le j\le 10}(1-\omega_jT),
$$
where
 $
M_r(X) =
\frac1{q^r}\bigl(N_r(X)-(
1+q^r+q^{2r}+q^{3r})\bigr) 
$ was defined in \eqref{mr}.

The reciprocity relation \eqref{rec} implies that
the polynomial $P_1(F(X),T)$ is   determined by the coefficients of $1,T,\dots,T^5$, hence by the numbers $N_1(X),\dots,N_5(X)$.\ The direct computation of these numbers is possible (with a computer) when $q$ is small (see Section \ref{secf2} for examples), but the amount of calculations quickly becomes very large. 

We will explain a method for  computing directly the   numbers $M_1(X),\dots,M_5(X)$.\ It was first introduced  in \cite{bsd}  and uses a classical geometric construction which expresses   the blow up of $X$ along a line as a conic bundle.\ It is valid only in \cha s $\ne 2$ and requires $X$ to contain  an $\F_q$-line $L$. 

Let $\widetilde X\to X$ be the blow up of $L$.\ Projecting from $L$ induces a morphism $\pi_L\colon \widetilde X\to \P^2_{\F_q}$ which is a conic bundle  and we denote by $\Gamma_L\subset \P^2_{\F_q}$ its discriminant curve, defined over $\F_q$.
 {\em Assume from now on that $q$ is odd;} the curve  $\Gamma_L$  is then a nodal plane quintic curve   and the associated double cover $\rho\colon \widetilde \Gamma_L\to \Gamma_L$ is admissible in the sense of \cite[D\'efinition 0.3.1]{bea} (the curve $\widetilde \Gamma_L$ is nodal and the fixed points of the   involution associated with $\rho$ are exactly the nodes of 
$\widetilde \Gamma_L$; \cite[Lemma~2]{bsd}).\footnote{In \cha\ 2, the curves  $  \Gamma_L$ and $ \widetilde \Gamma_L$  might not be nodal (see Lemma~\ref{fermat2}).} 

One can then define the Prym variety associated with $\rho$ and it  is isomorphic to the Albanese variety of the surface $F(X)$ (\cite[Theorem 7]{mur} when $\Gamma_L$ is smooth).\ The following is \cite[Formula (18)]{bsd}.

\begin{prop}\label{proposition number of points Bombieri}
Let $X\subset \P^4_{\F_q}$ be a smooth cubic threefold defined over $\F_q$, with $q$ odd, and assume that $X$ contains an $\F_q$-line $L$.\ With the notation \eqref{mr}, we have, for all $r\ge 1$,
$$M_r(X)=N_r(\widetilde \Gamma_L)-N_r( \Gamma_L).$$
\end{prop}

\begin{proof}
We will go quickly through the proof of \cite{bsd} because it is the basis of our algorithm.\ A point $x\in \P^2(\F_{q})$ corresponds to an $\F_q$-plane $P_x\supset L$ and the fiber $\pi_L^{-1}(x)$ is 
 isomorphic to  the conic $C_{x}$ such that
$X\cap P_{x}= L+C_x$.\ We have four cases:
\begin{itemize}
\item[(i)] either $C_{x}$ is geometrically irreducible, \ie, $x\not\in\Gamma_{L}(\F_q)$, in which case $\pi_L^{-1}(x)(\F_q)$ consists of $q+1$ points;
\item[(ii)] or $C_{x} $ is the union of two different
$\F_q$-lines, \ie, $x$ is smooth on $\Gamma_L$ and the $2$ points of $ \rho^{-1}(x) $  
are in $\widetilde \Gamma_L(\F_q)$, in which case $\pi_L^{-1}(x)(\F_q)$ consists of $2q+1$ points;
\item[(iii)] or $C_{x} $ is the union of two different conjugate
$\F_{q^2}$-lines, \ie, $x$ is smooth on $\Gamma_L$ and the $2$ points of $ \rho^{-1}(x) $  
are {\em not} in $\widetilde \Gamma_L(\F_q)$, in which case $\pi_L^{-1}(x)(\F_q)$ consists of $ 1$ point;
\item[(iv)] or $C_x$ is twice an $\F_q$-line, \ie, $x$ is singular on $\Gamma_{L}$, in  which case    $\pi_L^{-1}(x)(\F_q)$ consists of   $q+1$ points.
\end{itemize}
The total number of points of $\widetilde \Gamma_L(\F_{q})$ lying on a degenerate conic $C_x$ is therefore $q N_1(\widetilde \Gamma_L)+N_1( \Gamma_L)$ and we obtain
\[
N_1(\widetilde X)=(q+1)\bigl(N_{1}(\P^{2}_{\F_q})-N_{1}(\Gamma_{L})\bigr)+q N_1(\widetilde \Gamma_L)+N_1( \Gamma_L).
\]
Finally, since each point on $L\subset  X$ is replaced by a $\P^{1}_{\F_q}$
on $\widetilde X$, we have  
\[
N_{1}(\widetilde X)=N_{1}(X)-(q+1)+(q+1)^{2},
\]
 thus $N_{1}(X)=q^{3}+q^{2}+q+1+q\bigl(N_{1}(\widetilde \Gamma_L)-N_{1}(\Gamma_{L})\bigr)$.\ Since   the same conclusion holds upon replacing $q$ with $q^r$, this proves the proposition.
\end{proof}

Let $x\in \Gamma_L(\F_q)$.\  In order to compute
the numbers $N_1( \Gamma_L)-N_1(\widetilde\Gamma_{L})$, we need to understand
when the    points of $\rho^{-1}(x)$ are defined over $\F_q$. 

We follow \cite[p.~6]{bsd}.\ Take homogenous $\F_q$-coordinates $x_1,\dots,x_5$ on $\P^4$ so that $L$ is given by the equations $ x_{1}=x_{2}=x_{3}=0 $.\ The equation of the cubic $X$ can then be written as
\[
f+2q_{1}x_{4}+2q_{2}x_{5}+
 \ell_{1}x_{4}^{2}+2\ell_{2}x_{4}x_{5}+\ell_{3}x_{5}^{2}=0,
\]
 where $f$ is a cubic  form,
$q_{1},q_{2}$ are
quadratic forms, and  
 $\ell_{1},\ell_{2},\ell_{3}$ are linear forms  in the variables $x_{1},x_{2},x_{3}$.\ We choose the plane $\P^2_{\F_q}\subset \P^4_{\F_q}$ defined by $x_4=x_5=0$.\ If $x=(x_1,x_2,x_3,0,0)\in \P^2_{\F_q}$, the conic $C_x$ considered above is defined by the equation
$$
 f y_{1}^{2}  
 +2q_{1}y_{1}y_{2}+2q_{2}y_{1}y_{3} +
 \ell_{1}y_{2}^{2}+2\ell_{2}y_{2}y_{3}+\ell_{3}y_{3}^{2}=0
$$
and the quintic $\Gamma_L\subset \P^2_{\F_q}$ is defined by the equation
 $\det(M_L)=0$, where
\begin{equation}\label{ml}
M_L:=\begin{pmatrix} 
f & q_{1} & q_{2}  \\
 q_{1} & \ell_{1} & \ell_{2} \\
 q_{2} & \ell_{2} & \ell_{3} 
\end{pmatrix}. 
\end{equation}
For each $i\in\{1,2,3\}$, let $\delta_{i}\in H^{0}\bigl(\Gamma_{L},\cO(a_i)\bigr)$, where $a_i=2$, $4$, or $4$,
be the determinant of the submatrix of $M_L$ obtained by deleting its $i$th row and    $i$th column.\ The $-\delta_{i}$ are transition
functions of an invertible sheaf $\mathcal{L}$ on $\Gamma_{L}$ such
that $\mathcal{L}^{\otimes2}=\omega_{\Gamma_{L}}$ (a thetacharacteristic).\ It defines the double cover $\rho\colon \widetilde\Gamma_{L}\to\Gamma_{L}$.

A point $x\in \P^2_{\F_q}$ is singular on $\Gamma_L$ if and only if $\delta_1(x)=\delta_2(x)=\delta_3(x)=0$.\ These points do not contribute to $M_r$ since the only  point of $\rho^{-1}(x)$ is defined over the field of definition of $x$.\ This is the reason why we may assume that $x$ is smooth in the next proposition.

\begin{prop}
\label{delta}
Let $x$ be a smooth $\F_q$-point of $\Gamma_L$.\
The curve $\widetilde \Gamma_L$ has two $\F_q$-points over $x\in\Gamma_{L}(\F_q)$
if and only if either $-\delta_1(x)\in(\F_q^\times)^{2}$, or $ \delta_1(x)=0$ and either $-\delta_2(x)$ or $-\delta_3(x)$ is in 
 $ (\F_q^\times)^{2}.$\end{prop}

\begin{proof}
With the notation above, the line $L=V(y_{1})\subset \P^2_{\F_q}$ meets the conic  $C_{x}\subset \P^2_{\F_q}$
at the points $(0,y_2,y_3)$ such that 
\[
 \ell_{1}y_{2}^{2}+2\ell_{2}y_{2}y_{3}+\ell_{3}y_{3}^{2}=0.
\]
Therefore, if $-\delta_1(x)=\ell_2^{2}(x)-\ell_{1}(x)\ell_3(x)$
is nonzero, the curve $\widetilde \Gamma_L$ has two rational points over $x\in\Gamma_{L}(\F_q)$
if and only if $-\delta_1(x)\in (\F_q^\times)^{2}$. 

 When $\delta_1(x)=0$, 
we have $C_x= L_{1}+L_{2}$,
where $L_{1}$ and $L_{2}$ are lines 
meeting in an $\F_q$-point  $z$ of $L$ which we   assume to be $(0,0,1)$.\ This means that there is no $y_3$ term in the equation of $C_x$, hence $\ell_{2}(x)=\ell_{3}(x)=q_2(x)=0$.\ The conic $C_x$ is defined by  the equation
$$
 \ell_{1}(x)y_{2}^{2}+2q_{1}(x)y_{1}y_{2}+ f (x)  y_{1}^{2} =0
$$
and the two lines $L_{1}$ and $L_{2}$ are   defined over $\F_q$ if and only if $-\delta_3(x)=q_1^{2}(x)-\ell_{1}(x)f(x)\in (\F_q^\times)^{2}$ (since $\delta_1(x)=\delta_2(x)=0$, this is necessarily nonzero because $x$ is smooth on  $\Gamma_L$).

For the general case: if $y_3(z)\ne 0$, we make a linear change of coordinates $y_1=y'_1$, $y_2=y'_2+ty'_3$, $y_3=y'_3$ in order to obtain $y'_2(z)=0$, and we check that  $-\delta_3(x)$ is unchanged; if $z= (0,1,0)$, we obtain as above   $\delta_1(x)=\delta_3(x)=0$ and   $L_{1}$ and $L_{2}$ are   defined over $\F_q$ if and only if $-\delta_2(x) \in (\F_q^\times)^{2}$.\ This proves the proposition.
\end{proof}

We can now describe our algorithm for the computation of the numbers $M_{r}(X)=N_{r}(\widetilde\Gamma_{L})-N_{r}(\Gamma_{L})$.

The input data is a cubic threefold $X$ over $\F_q$ containing
an $\F_{q}$-line $L$.\ We choose coordinates as above and construct the matrix $M_L$ of \eqref{ml} whose determinant is the equation
of the quintic $\Gamma_{L}\subset \P^{2}_{\F_q}$.\ We compute $M_{r} $
with the following simple algorithm.

\begin{algorithm}[H]\label{algo}
 \KwIn{$(X,L,r)$}
 \KwOut{$M_r$}
 Compute the matrix $M_L$, the three minors $\delta_1,\delta_2,\delta_3$ and
 the curve $\Gamma_L$\;
 $M_r := 0$\;
 \While{$p\in \{p\, :\, p \in \Gamma_L(\F_{q^r})\mid \Gamma_L\text{ is smooth at }p\}$}{
  \eIf{$-\delta_1(p)\in(\F_{q^r}^\times)^2$ or 
  $(\delta_1(p) = 0$ and 
 $(-\delta_2(p)\in(\F_{q^r}^\times)^2 \text{ or } -\delta_3(p)\in(\F_{q^r}^\times)^2))$}{
   $M_r := M_r + 1$\;
   }{
   $M_r := M_r-1$\;
  }
 }
 \KwRet{$M_r$}\;
\caption{Computing $M_r$.}
\end{algorithm}

 \subsection{Lines on mildly singular  cubic threefolds}\label{nnodal} 
We describe a method based on results of Clemens--Griffiths and Kouvidakis--van der Geer which reduces the computation of the number of $\F_q$-lines on a cubic  threefold   with a single singular point, of type $A_1$ or  $A_2$,\footnote{A hypersurface singularity is of type $A_j$ if it is, locally analytically, given by an equation $x_1^{j+1}+x_2^2+\dots+x_{n+1}^2=0$.\ Type $A_1$ is also called a node.}  to the computation of the number of points on a smooth curve of genus 4.\ One consequence is that there is always an $\F_q$-line when $q>3$. 

Let $C$ be a smooth nonhyperelliptic curve of genus 4 defined over a perfect field $\F$.
We denote by $g^1_3$ and $h^1_3=K_C-g^1_3$ the (possibly equal) degree-3 pencils on $C$.\ The canonical curve $\phi_{K_C}(C)\subset \P^3_\F$ is contained in a unique geometrically integral   quadric surface $Q$ whose rulings cut out the degree-3 pencils on $C$; more precisely,
\begin{itemize}
\item either $Q\isom \P^1_\F\times \P^1_\F$ and the two rulings of $Q$ cut out distinct degree-3 pencils $g^1_3$ and $h^1_3=K_C-g^1_3$ on $C$ which are defined over $\F$;
\item or $Q$ is smooth but its two  rulings are defined over a quadratic extension of $\F$ and are exchanged by the Galois action, and so are $g^1_3$ and $h^1_3$;
\item or $Q$ is singular and its ruling cuts out   a degree-3 pencil  $g^1_3$  on $C$ which is defined over $\F$ and satisfies $K_C=2g^1_3 $.
\end{itemize}
Let  
$\rho\colon \P^3_\F\dra\P^4_\F$ be the rational map defined by the 
linear system of cubics containing $\phi_{K_C}(C)$.\ The image 
of $\rho$ is a cubic threefold $X$ defined over $\F$; it has a single   singular point, 
  $\rho(Q)$, which is of type $A_1$ if $Q$ is smooth, and of type $A_2$ otherwise.\ Conversely, every  cubic  threefold  $X\subset \P^4_\F$  defined over $\F$ with a single singular point $x$, of type $A_1$ or  $A_2$,  is obtained in this fashion: the curve $C$ is $\T_{X,x}\cap X$ and parametrizes the lines in $X$ through $x$ (\cite[Corollary~3.3]{cml}).

The   surface $F(X)$ is isomorphic
to the nonnormal surface 
obtained
by gluing the   images $C_g$ and $C_h$ of   the morphisms $C\to C^{(2)}$  defined by
$p\mapsto g_3^1-p$ and $p\mapsto h_3^1-p$ (when $Q$ is singular,   $F(X)$ has a 
cusp singularity along the curve $C_g=C_h$).\ This was proved in  \cite[Theorem~7.8]{cg} over $\C$ and in \cite[Proposition 2.1]{kvdg} in general.

\begin{prop}\label{propnodal}
Let $X\subset \P^4_{\F_q}$ be a cubic  threefold defined over ${\F_q}$ with a single singular point, of type $A_1$ or  $A_2$.\ Let $C$ be the associated curve of genus $4$, with degree-3 pencils $g_3^1$ and $h_3^1$.\ For any $r\ge 1$,   set $n_r := \Card  (C(\F_{q^r}) )$.\ We have
$$
 \Card\bigl(F(X)(\F_q)\bigr) = 
 \begin{cases}
 \frac12 (n_1^2-2n_1+n_2) & \text{ if     $g_3^1$ and $h_3^1$ are distinct  and defined over $\F_q$;}\\
 \frac12 (n_1^2+2n_1+n_2) & \text{ if  $g_3^1$ and $h_3^1$ are not defined over $\F_q$;}\\
 \frac12 (n_1^2+n_2)& \text{ if  $g_3^1=h_3^1$.}  \end{cases}
$$
\end{prop}

\begin{proof}
Points of $ C^{(2)}(\F_q)$  correspond to
\begin{itemize}
\item the $ \frac12 (n_1^2- n_1)$ pairs of distinct  points of $C(\F_q)$,
\item the $  n_1 $    $\F_q$-points on the diagonal,
\item the $ \frac12 (n_2 - n_1)$ pairs of distinct conjugate points of $C(\F_{q^2})$,
\end{itemize}
for a total of $ \frac12 (n_1^2 +n_2) $ points (compare with \cite[(2.5)]{gash}).\ 
When $g_3^1$ and $h_3^1$ are distinct  and defined over $\F_q$, the gluing process eliminates $n_1$ $\F_q$-points.\ When $g_3^1$ and $h_3^1$ are not defined over $\F_q$, the curves $C_g$ and~$C_h$ contain no pairs of conjugate points, and the gluing process creates $n_1$ new $\F_q$-points.\ Finally, when $g_3^1=h_3^1$, the map $C^{(2)}(\F_q)\to F(X)(\F_q)$ is a bijection.
\end{proof}

\begin{coro}\label{coronodal}
When $q\ge 4$, any cubic  threefold $X\subset \P^4_{\F_q}$   defined over ${\F_q}$ with a single singular point, of type $A_1$ or  $A_2$, contains an $\F_q$-line.
\end{coro}

For $q\in\{2,3\}$, we produce in Section \ref{nodal} explicit examples of cubic threefolds with a single singular point, of type $A_1$, but containing no $\F_q$-lines: the  bound in the corollary is the best possible.

\begin{proof}
Assume that $X$ contains no $\F_q$-lines.\ Proposition \ref{propnodal} then implies that either $n_1=n_2=0$, or $n_1=n_2=1$ and $g_3^1$ and $h_3^1$ are distinct  and defined over $\F_q$.\ The latter case cannot in fact occur: if $C(\F_q)=\{x\}$, we write $g^1_3\equiv x+x'+x''$.\ Since $g_3^1$ is defined over $\F_q$, so is $x'+x''$, hence $x'$ and $x''$ are both defined over $\F_{q^2}$.\ But  $C(\F_{q^2})=\{x\}$, hence $x'=x''=x$ and $g^1_3\equiv 3x$.\ We can do the same reasoning with $h^1_3$ to obtain $h^1_3\equiv 3x\equiv g^1_3 $, a  contradiction.

Therefore, we have $n_1=n_2=0$.\ According to  \cite[Theorem 1.2]{HLT}, every genus-4 curve over $\F_q$ with $q>49$ has an $\F_q$-point so we obtain $q\le7$. 

Because of the     reciprocity relation \eqref{rec}, there is a  monic degree-4 polynomial $H$ with integral coefficients that satisfies $Q_1(C,T) = T^4  H(T + q/T)$.\
 If $\omega_1,\dots,\omega_4,\bar \omega_1,\dots,\bar\omega_4$ are the roots of $Q_1(C,T)$ (see Section \ref{sso}), with $|\omega_j|=\sqrt{q}$, the roots of $H $ are the $r_j:=\omega_j+\bar \omega_j$, and
$$q+1-n_1=\sum_{1\le j\le 4}r_j\quad,\quad q^2+1-n_2=\sum_{1\le j\le 4}(\omega_j^2+\bar \omega_j^2)=\sum_{1\le j\le 4}(r_j^2-2q).$$
 Since $n_1=n_2=0$, we obtain
 $
 \sum_{1\le j\le 4}r_j=q+1$ and $
\sum_{1\le j\le 4} r_j^2 = q^2+8q+1$, so that
  $\sum_{1\le i<j\le 4} r_ir_j = -3q$;   we can therefore write
\begin{equation}
\label{polF}
 H(T) = T^4 - (q +1)T^3 - 3qT^2 + aT + b.
\end{equation}   
 Finally, since $|r_j|\le2\sqrt q$ for each $j$, we also have $|b|=|r_1r_2r_3r_4|\le 16q^2$ and
  $|a|=|\sum_{j=1}^4b/r_j|\leq 32q^{3/2}$.
A computer search  done with these bounds  
shows that   polynomials of the form~\eqref{polF} 
with four real roots and $q\in\{2,3,4,5,7\}$ only exist for $q\le 3$, which proves the corollary. \end{proof}

\begin{rema}\upshape
For $q\in\{2,3\}$, the computer  gives a list of all   
  possible polynomials   
$$
(q=2)\quad H(T)=\left\{\begin{smallmatrix}  T^4 - 3T^3 - 6T^2 + 24T - 16\\
     T^4 - 3T^3 - 6T^2 + 24T - 15&(\star)\\
     T^4 - 3T^3 - 6T^2 + 23T - 13\\
     T^4 - 3T^3 - 6T^2 + 22T - 10&(\star)\\
     T^4 - 3T^3 - 6T^2 + 21T - 7&\\
     T^4 - 3T^3 - 6T^2 + 18T + 1&(\star)\end{smallmatrix}\right. ,\quad
(q=3)\quad H(T)=\left\{\begin{smallmatrix}
  T^4 - 4T^3 - 9T^2 + 48T - 36&(?)\\
     T^4 - 4T^3 - 9T^2 + 47T - 32&(\star)\\
     T^4 - 4T^3 - 9T^2 + 46T - 29&(\star)\\
     T^4 - 4T^3 - 9T^2 + 44T - 22&(?)
\end{smallmatrix}\right. .$$
The nodal cubics of Section~\ref{nodal}, defined over $\F_2$ and
$\F_3$,  correspond to the 
polynomials $T^4 - 3T^3 - 6T^2 + 24T - 15$ and 
$T^4 - 4T^3 - 9T^2 + 47T - 32$, respectively.\ Over $\F_2$, it is possible to list all   genus-4 canonical curves  and one obtains that only the   polynomials marked with   $(\star)$ actually occur (all three are irreducible). 

Over $\F_3$, our computer searches show that the two polynomials marked with   $(\star)$ actually occur (both are irreducible).\ We do not know whether the other two, $T^4 - 4T^3 - 9T^2 + 48T - 36 = (T - 1)(T - 3)(T^2 - 12)$ and $
 T^4 - 4T^3 - 9T^2 + 44T - 22 = (T^2 - 4T + 2)(T^2 - 11)$ (marked with (?)), actually occur.
 \end{rema}

\subsection{Examples of cubic threefolds}\label{secf2}

In this section, we present some of our calculations and illustrate our techniques for some cubic threefolds.\ We begin with    Fermat cubics  
(Section \ref{exfermat}), which have good reduction in all \cha s but 3.\ The case of general Fermat hypersurfaces was worked out   by Weil in \cite{weil} (and was an inspiration for his famous conjectures discussed in Section \ref{sedwc}).\ We explain how Weil's calculations apply to the zeta function of Fermat cubics (Theorem \ref{thweil}) and we compute, in dimension 3,   the zeta function of their surface of lines (Corollary \ref{coroweil}). 

The Fermat cubic threefold  contains the line $L:=\langle  (1 , -1 , 0 ,0,0),  (0, 0 , 1 , -1 , 0) \rangle$ and we  compute the discriminant quintic $\Gamma_L\subset \P^2$ defined in Section \ref{secbsd}, exhibiting strange behavior in \cha\ 2.  

In Section \ref{kleins}, we turn our attention to the Klein cubic, which has good reduction in all \cha s but 11.\ It also contains an ``obvious'' line $L'$ and we compute the discriminant quintic $\Gamma_{L'}\subset \P^2$, again exhibiting strange behavior in \cha\ 2.\ Using the Bombieri--Swinnerton-Dyer method, we determine the zeta function of $F(X)$ over $\F_p$, for $p\le  13$.\ We also compute the geometric Picard numbers of the reduction of $F(X)$ modulo any prime, using the existence of an isogeny between $A(F(X))$ and the self-product of an elliptic curve.

 In Section \ref{algos}, we compute, using the same method, the zeta function of $F(X)$ of a ``random" cubic threefold $X$ containing a line, over the fields $\F_5$, $\F_7$,   $\F_{23}$,   $\F_{29}$, and $\F_{31}$.\ Note that existing programs are usually unable to perform calculations in such   high \cha s.

In Section \ref{noli}, we present   examples,  found by computer searches, of smooth   cubic threefolds defined over $\F_2$, $\F_3$, $\F_4$, or $\F_5$ with no lines.\ We were unable to find examples over $\F_q$   for the remaining values $q\in\{7,8,9\}$ (by Theorem \ref{prop6},  there are always $\F_q$-lines for $q\ge 11$).
 For the example over $\F_2$, we compute directly the number of points over small extensions and deduce the   polynomial $P_1$ for the Fano surface $F(X)$.\ For the example over $\F_3$, we obtain again   the   polynomial $P_1$ for the Fano surface $F(X)$ by applying the Bombieri--Swinnerton-Dyer method over $\F_9$. 
 
 Finally, in Section \ref{nodal}, we exhibit  cubic threefolds with one node but no lines, defined over $\F_2$ or $\F_3$, thereby proving that the bound in Corollary \ref{coronodal} is optimal.

\subsubsection{Fermat cubics}\label{exfermat}

The $n$-dimensional Fermat cubic $X^n\subset \P^{n+1}_{\Z}$ is defined  by the equation
\begin{eqnarray}\label{equ2}
x_1^3 + \dots+x_{n+2}^3=0. 
\end{eqnarray}
It has good reduction at every prime $p\ne 3$.

\begin{rema}\label{cyc}\upshape
In general, if 
$q \equiv 2\pmod 3$ and $X\subset \P^{n+1}_{\F_q}$  is a cyclic cubic hypersurface 
defined by the equation $ f(x_1,\dots,x_{n+1})+x_{n+2}^3 =0$,   the projection  
$\pi \colon  X\to\P^n_{\F_q}$ defined by $(x_1,\dots,x_{n+2})
\mapsto (x_1,\dots,x_{n+1})$ 
induces a bijection $X(\F_q)\to \P^n(\F_q)$, because the map $x\mapsto x^3$ is a bijection of $\F_q$ (\cite[Observation 1.7.2]{ked}).
 \end{rema}

The remark gives in particular $\Card\bigl(X^n(\F_2)\bigr)=\Card\bigl( \P^n(\F_2)\bigr)=2^{n+1}-1$.\ For the number of points of $X^n(\F_4)$,  observe that the cyclic cover $\pi$ is  3-to-1 outside its branch divisor 
$V(f)$.\ Let $ (x_1,\dots,x_{n+1})\in \P^n(\F_4)$.\ Since $x^3\in \{0,1\}$ for any $x\in
\F_4$, either $x_1^3+\dots+x_{n+1}^3=0$ and the inverse image by $\pi$ has one $\F_4$-point, or   $x_1^3+\dots+x_{n+1}^3=1$ and the inverse image by $\pi$ has three $\F_4$-points.\ One obtains the inductive formula
$$\Card\bigl(X^n(\F_4)\bigr) =\Card\bigl(X^{n-1}(\F_4)\bigr)+ 3 \bigl(\Card\bigl(\P^n(\F_4)\bigr)-\Card\bigl(X^{n-1}(\F_4)\bigr)\bigr)
.$$
Since $\Card\bigl(X^0(\F_4)\bigr)=3$, we get
\begin{equation*} 
 \Card\bigl(X^n(\F_4)\bigr) = \frac{1}{3}\left(2^{2n+3} - (-2)^{n+1} - 1 \right).
\end{equation*}

Using \eqref{gsf}, we see that the number of $\F_2$-lines on   $X^n_{\F_2}$  is
\[
 \frac{(2^{n+1}-1)^2 - 2 (1+2^n)(2^{n+1}-1)
 + \frac{1}{3}(2^{2n+3} - (-2)^{n+1} - 1 )}{8}
 =
 \frac{2^{2n}+1+((-1)^n-9)2^{n-2}}{3}.
\]
For example,  the  15 $\F_2$-lines contained in $X^3_{\F_2}$ are the line $L_{\F_2}$ and its images by permutations of the coordinates.

In fact,     general results are available in the literature on   the zeta function  of Fermat hypersurfaces over finite fields (starting with \cite{weil}; see also \cite[Section 3]{shka}), although they do not seem to have been spelled out for cubics.\ Let us first define
$$
P^0_n(X^n_{\F_p},T)=  \begin{cases} P_n(X^n_{\F_p},T)&\quad\hbox{if $n$ is odd,}\\
\frac{P_n(X^n_{\F_p},T)}{1-p^{n/2}T}&\quad\hbox{if $n$ is even}
\end{cases}
$$
(this is the reciprocal characteristic polynomial of the Frobenius morphism acting on the {\em primitive} cohomology of $ X^n_{\F_p}$)  and set $
b_n^0(X^n):=\deg(P^0_n )$; this is $b_n(X^n)$ if $n$ is odd, and $b_n(X^n)-1$ if $n$ is even.
 
\begin{theo}[Weil]\label{thweil}
Let    $X^n\subset \P^{n+1}_\Z$ be the Fermat cubic hypersurface.\ Let $p$ be a prime number other than $3$. 
\begin{itemize}
\item If $p \equiv 2\pmod 3$, we have
$$P^0_n(X^n_{\F_p},T)=(1-(-p)^nT^2)^{b^0_n(X^n)/2}.
$$
\item If $p \equiv 1\pmod 3$,  one can write uniquely  $4p=a^2+27b^2$ with $a\equiv 1\pmod 3$ and $b> 0$, and
$$P^0_n(X^n_{\F_p},T)=  \begin{cases}
1+aT+pT^2&\quad\hbox{when $n=1$,}\\
(1-pT)^6&\quad\hbox{when $n=2$,}\\
(1+apT+p^3T^2)^5&\quad\hbox{when $n=3$,}\\
(1-p(2p-a^2)T+p^4T^2)(1-p^2T)^{20}&\quad\hbox{when $n=4$,}\\
 (1+ap^2T+p^5T^2)^{21}&\quad\hbox{when $n=5$.}
\end{cases}$$
 \end{itemize}
\end{theo}

As will become clear from the proof, it would be possible to write down  (complicated) formulas for all $n$ in the case $p \equiv 1\pmod 3$.\ We leave that exercise to the interested reader and restrict ourselves to the lower-dimensional cases.

\begin{proof}
Assume first  $p \equiv 2\pmod 3$.\  It follows from Remark \ref{cyc} that the polynomial $P^0_n(X^n_{\F_p},T)$ is even (this is explained by \eqref{mr4} and \eqref{mr44} when $n=4$).\ It is therefore equivalent to prove $P^0_n(X^n_{\F_{p^2}},T)=(1-(-p)^nT )^{b^0_n(X^n) }$.\ We follow the geometric argument of \cite{shka}.

It is well known that $P_1(X^1_{\F_p},T)=1+pT^2$, hence $P_1(X^1_{\F_{p^2}},T)=(1+pT)^2$.\ In other words, the Frobenius morphism of $\F_{p^2}$ acts on the middle cohomology   of $X^1_{\F_{p^2}}$ by multiplication by~$-p$.\ By the K\"unneth formula, it acts by multiplication by $ (-p)^2$ 
on the middle cohomology   of $X^1_{\F_{p^2}}\times X^1_{\F_{p^2}}$.\ The proof by induction on $n$ of \cite[Theorem 2.10]{shka}  then applies and gives that the Frobenius morphism acts by   multiplication by $ (-p)^n$ 
on the middle cohomology   of $X^n_{\F_{p^2}}$.

Assume now  $p \equiv 1\pmod 3$.\ The number of points of $X^1(\F_p)$ was computed by Gauss (\cite[Theorem 4.2]{sita}): writing   $4p=a^2+27b^2$ as in the theorem,  one has $\Card (X^1(\F_p) )=p+1+a$, \ie, $P_1(X^1_{\F_p},T)=1+aT+pT^2=:(1-\omega T)(1-\bar\omega T)$.\ 
In other words, the  eigenvalues of the Frobenius morphism of $\F_p$ acting on the first cohomology group are $\omega$ and $\bar \omega$.\ They are therefore the Jacobi sums denoted by $j(1,2)$ and $j(2,1)$ in \cite[(3.1)]{shka}, and also the generators of the prime ideals $\gp$ and $\bar \gp$ in 
 $\Z[\zeta]$ ($\zeta= \exp(2i\pi/3)$) such that $(p)=\gp \bar \gp$.
 
 The eigenvalues of the Frobenius morphism   acting on the   primitive  middle cohomology   of $X^n_{\F_p}$ are denoted $j(\alpha)$ by Weil, where $\alpha$ runs over the set
 $$\gU_n=\{(\alpha_0,\dots,\alpha_{n+1})\in \{1,2\}^{n+2}\mid \alpha_0+\dots+\alpha_{n+1} \equiv 0\pmod3\}.$$
 The   ideal $(j(\alpha))$ in   $\Z[\zeta]$ is invariant under permutations of the $\alpha_i$ and its decomposition is   computed by Stickelberger (see \cite[(3.10)]{shka}):
 $$(j(\alpha))=\gp^{A(\alpha)} \bar\gp^{A(\bar\alpha)},$$
 with $A(\alpha)=\bigl\lfloor\sum_{j=1}^{n+1}  \frac{\alpha_j}{3} \bigr\rfloor$ and
 $\bar\alpha_j=3-\alpha_j$.   
  
The elements of $\gU_1$ are $(1,1,1)$ and $(2,2,2)$, and the corresponding values of  $A$ are $0$ and $1$.\ The eigenvalues are therefore (up to multiplication by a unit of  $\Z[\zeta]$), $\omega$ and $\bar\omega$.\ By Gauss' theorem, we know they are exactly $\omega$ and $\bar\omega$.\ By induction on $n$, it   then follows  from   the embeddings  \cite[(2.17)]{shka} that 
$$ j(\alpha) =\omega^{A(\alpha)} \bar\omega^{A(\bar\alpha)}.$$

The elements of $\gU_2$ are (up to permutations) $( 1,1,2,2)$  and the corresponding value  of  $A$ is~$1$.\ The only eigenvalue is therefore  $\omega \bar\omega=p $, with multiplicity $\binom42$.  

The elements of $\gU_3$ are (up to permutations) $( 1,1,1,1,2)$  and $( 1,2,2,2,2)$, and the corresponding values of  $A$ are $1$ and $2$.\ The eigenvalues are therefore   $\omega^2 \bar\omega=p \omega$ and $p\bar\omega$, with multiplicity~$5$.  

The elements of $\gU_4$ are (up to permutations) $( 1,1,1,1,1,1)$, $( 1,1,1,2,2,2)$,\break and $( 2,2,2,2,2,2)$, and the corresponding values of  $A$ are $1$, $ 2$, and $3$.\ The eigenvalues are therefore  $p\omega^2  $ and $p\bar\omega^2  $,  with multiplicity $1$, and $p^2$,  with multiplicity $\binom63$.

The elements of $\gU_5$ are (up to permutations) $( 1,1,1,1,1,2,2)$ and $( 1,1,2,2,2,2,2)$,  and the corresponding values of  $A$ are  $ 2$  and $3$.\ The eigenvalues are therefore   $p^2\omega   $ and $p^2\bar\omega $, with multiplicity~$\binom72$.\ This finishes the proof of the theorem.
 \end{proof}

\begin{coro}\label{coroweil}
Let    $X\subset \P^4_\Z$ be the Fermat cubic threefold defined by the equation $x_1^3+\cdots+x_5^3=0$ and let $F(X)$ be its surface of lines.\ Let $p$ be a prime number other than $3$. 

The Albanese variety $A(F(X))_{\F_p}$ is isogenous to $E_{\F_p}^5$, where $E$ is the Fermat plane cubic curve.\ Moreover,
\begin{itemize}
\item if $p \equiv 2\pmod 3$, we have  
$$Z(F(X)_{\F_p},T)=\frac{(1+pT^2)^5(1+p^3T^2)^5}{(1-T)(1-p^2T)(1+pT)^{20}(1-pT)^{25} }
,$$
the Picard number of $F(X)_{\F_p}$ is 25 and that of $F(X)_{\F_{p^2}}$ is 45, and the abelian variety $A(F(X))_{\F_p}$ is supersingular;
\item if $p \equiv 1\pmod 3$, we have (with the notation of Theorem \ref{thweil})
$$Z(F(X)_{\F_p},T)=\frac{(1+aT+pT^2)^5(1+apT+p^3T^2)^5}{(1-T)(1-p^2T)(1+(2p-a^2)T+p^2T^2)^{10}(1-pT)^{25}  }
,$$
the Picard and absolute Picard numbers of $F(X)_{\F_p}$ are 25, and the abelian variety $A(F(X))_{\F_p}$ is ordinary.
\end{itemize}
\end{coro}

\begin{proof} Theorems \ref{main} and \ref{thweil} imply that the characteristic polynomials of the Frobenius morphisms acting on $H^1$ are the same for the   abelian varieties $A(F(X))_{\F_p}$ and $E_{\F_p}^5$; they are therefore isogenous (\cite[Appendix I, Theorem 2]{MumfordAV}).\ The statements about $A(F(X))_{\F_p}$ being supersingular or ordinary follow from the analogous statements about $E_{\F_p}$.

The values of the zeta functions also follow from Theorem \ref{main} and \ref{thweil},  and  the statements about the  Picard numbers from Corollary \ref{pic}. \end{proof}

 {\em We now restrict ourselves to the  Fermat cubic threefold $X\subset \P^4_\Z$} ($n=3$).
 With the notation of Theorem \ref{thweil}, we compute   (see Remark \ref{atc})
  \begin{equation}\label{atc1}
D_p\bigl(F(X)_{\F_p}\bigr)=\begin{cases}
2^{20} &\hbox{ if } p \equiv 2\pmod 3;\\
\frac{3^{30}b^{20}}{p^{10}}&\hbox{ if } p \equiv 1\pmod 3.
\end{cases}\end{equation}

Next, we parametrize planes containing the line $L:=\langle  (1 , -1 , 0 ,0,0),  (0, 0 , 1 , -1 , 0) \rangle\subset X$ by the   $\P^2$ defined by $x_1=x_3=0$ and determine the discriminant quintic $\Gamma_L\subset \P^2$ (see Section~\ref{secbsd}).

\begin{lemm}\label{fermat2}
In the   coordinates $x_2,x_4, x_5$, an equation of the discriminant quintic $\Gamma_L\subset \P^2$ is $x_2x_4(x_2^3+x_4^3+4x_5^3)=0$.\ Therefore,
\begin{itemize}
\item in \cha s other than 2 and 3,  it is a nodal quintic which is the union of two lines and an elliptic curve, all defined over the prime field; 
\item in \cha\ 2, it  is the union of 5 lines meeting at the point $(0,0,1)$; 3 of them are defined over $\F_2$, the other 2 over $\F_4$.
\end{itemize}
\end{lemm}

\begin{proof}
We use the notation of the proof of Proposition \ref{proposition number of points Bombieri} (although the choice of coordinates is different).\ If $x=(0,x_2,0,x_4,x_5)\in \P^2$, the residual conic $C_x$ is defined by  the equation 
$$\frac{1}{y_1}\bigl(y_2^3+(x_2y_1-y_2)^3+y_3^3+(x_4y_1-y_3)^3+y_1^3x_5^3\bigr)=y_1^2(x_2^2+x_4^2+x_5^2)-3x_2^2y_1y_2-3x_4^2y_1y_3+3x_2y_2^2+3x_4y_3^2 
$$
 in the coordinates $(y_1,y_2,y_3)$.\ In \cha s other than 2 and 3, an equation of $\Gamma_L$ is therefore given by
$$\left|\begin{smallmatrix} 
x_2^3+x_4^3+x_5^3 & -\frac32 x_2^2 & -\frac32 x_4^2 \\
-\frac32 x_2^2 & 3x_2& 0 \\
-\frac32 x_4^2& 0 & 3x_4 
\end{smallmatrix}\right|
= \frac94 x_2x_4\left|\begin{smallmatrix} 
4(x_2^3+x_4^3+x_5^3) &  3x_2 & 3x_4 \\
 x_2^2 & 1 & 0 \\
 x_4^2& 0 & 1 
\end{smallmatrix}\right|
=\frac94 x_2x_4(x_2^3+x_4^3+4x_5^3)=0.
$$
In \cha\ 2,  the Jacobian criterion says that the singular points of $C_x$ must satisfy $y_1=0$ and $x_2y_2^2+x_4y_3^2=x_2^2y_2+x_4^2y_3=0$.\ The  curve $\Gamma_L$  is therefore defined by $\left|\begin{smallmatrix} 
x_2^{1/2}  & x_4^{1/2} \\
x_2^2 & x_4^2
\end{smallmatrix}\right| =0$, or $x_2x_4(x_2^3+x_4^3)=0$.\ It is therefore the ``same'' equation reduced modulo 2.
\end{proof}

\subsubsection{The Klein threefold}\label{kleins}
 This is the cubic threefold   $X\subset \P^4_{\Z}$  defined  by  the equation
\begin{eqnarray}\label{equ3}
x^2_1x_2 + x_2^2x_3 + x^2_3x_4 + x^2_4x_5 + x_5^2x_1=0. 
\end{eqnarray}
It has good reduction at every prime $p\ne 11$.

It  contains the line ${L'}= \langle  (1 , 0 , 0 ,0,0),  (0, 0 , 1 , 0 , 0) \rangle $ and we parametrize planes containing ${L'}$ by the   $\P^2$ defined by $x_1=x_3=0$.  

\begin{lemm}\label{klein2}
In the   coordinates $x_2,x_4, x_5$, an equation of the discriminant quintic $\Gamma_{L'}\subset \P^2$ is $x_2^5+x_4x_5^4-4x_2x_4^3x_5=0$.\
Therefore,
\begin{itemize}
\item in \cha s other than 2 and 11,  it is a geometrically irreducible quintic with a single singular point, $(0,1,0)$, which is a node; 
\item in \cha\ 2, it  is a geometrically irreducible  rational quintic with a single singular point of multiplicity 4, $(0,1,0)$.
\end{itemize}
\end{lemm}

\begin{proof}
We proceed as in the proof of Lemma \ref{fermat2}.\ If $x=(0,x_2,0,x_4,x_5)\in \P^2$, an    equation of the residual conic $C_x$ is 
$$\frac{1}{y_1}\bigl(y_2^2x_2y_1+x_2^2y_1^2y_3+ y_3^2x_4y_1+x_4^2y_1^2x_5y_1+x_5^2y_1^2y_2\bigr)=y_2^2x_2 +x_2^2y_1y_3+ y_3^2x_4+x_4^2y_1^2x_5+x_5^2y_1y_2 
$$
in the coordinates $(y_1,y_2,y_3)$.\ In \cha\ other than 2, an equation of $\Gamma_{L'}$ is therefore  
$$
\left|\begin{matrix} 
x_4^2 x_5 &  \frac12 x_5^2 &  \frac12 x_2^2 \\
 \frac12 x_5^2 &  x_2& 0 \\
 \frac12 x_2^2& 0 &  x_4 
\end{matrix}\right| =  \frac14   (x_2^5+x_4x_5^4-4x_2x_4^3x_5)=0.
$$
In \cha\ 2,  one checks that  $\Gamma_{L'}$  is defined by the  equation $x_2^5+x_4x_5^4=0$.\ In both cases, the singularities are easily determined.  
\end{proof}

 In \cha\ 11, $X_{\F_{11}}$ has   a unique singular point, $  (1 , 3 , 3^2 , 3^3 , 3^4)$, which has type $A_2$.\ The quintic $\Gamma_{L'}\subset \P^2$ is still geometrically irreducible,    with a node at $(0,1,0)$ and an ordinary cusp (type $A_2$) at  $(5,1,3)$.
 
In \cha\ 2, the isomorphism $(x_1,\dots,x_5)\mapsto
(x_1 + x_5,x_2 + x_5,x_3 + x_5,x_4 + x_5,
x_1+x_2+x_3+x_4+x_5)$ maps $X_{\F_2}$ to the
cyclic cubic defined by $ x_5^3 + (x_1+x_2+x_3+x_4)^3 
+ x_1^2x_2 + x_2^2x_3 + x_3^2x_4=0$.\  
Thus $M_{2m+1}(X_{\F_2}) = 0$ for any $m\ge0$ (reasoning
as in Example \ref{exfermat}).\
The computer gives $ M_2 (X_{\F_2}) = M_4(X_{\F_2})   =0$.\ Using \eqref{gsf}, we find that $X_{\F_2}$ contains 5 $\F_2$-lines; they are the line ${L'}$ and its images by the cyclic permutations of the coordinates.

By the reciprocity property \eqref{rec}, we obtain
$$P_1(F(X)_{\F_2},T)=P_3(X_{\F_2},T/2)=1+2^5T^{10}.$$
Since this polynomial has simple roots,   the Picard number of $F(X)_{\F_2}$ is $5$ (Corollary \ref{pic}).\ The eigenvalues of the Frobenius morphism $F$  are $\omega \exp( 2ik\pi/10)$, for $k\in\{0,\dots,9\}$, where $\omega^{10}=-2^5$; hence  $F^{10}$ acts by multiplication by $-2^5$.\  This implies $P_1(F(X)_{\F_{2^{10}}},T)= (1+2^5T)^{10}$.\ It follows that $F(X)_{\F_{2^{10}}}$ has maximal Picard number 45 (Corollary \ref{pic}) and that $A(F(X)) $ is  isogenous to~$E^5 $ over $\F_{2^{10}}$, where $E$ is the Fermat plane cubic defined in Section~\ref{exfermat}.

We also get $P_2(F(X)_{\F_2},T)=(1-2^5T^5)(1-2^{10}T^{10})^4=(1-2^5T^5)^5(1+2^5T^5)^4$ and
$$Z(F(X)_{\F_2},T)=\frac{(1+2^5T^{10})(1+2^{15}T^{10})}{(1-T)(1-4T)(1-2^5T^5)^5(1+2^5T^5)^4 }
.$$
 This implies
 (see Remark \ref{atcsrem})
 \begin{equation}\label{atc2}
D_2(F(X)_{\F_2})=\lim_{s\to1}\frac{P_{2}(F(X)_{\F_2},2^{-s})}{(1-2^{1-s})^{5}}=2^{4}\cdot  5^5.
\end{equation}

Over other small fields, we find, using the Bombieri--Swinnerton-Dyer method (Proposition~\ref{proposition number of points Bombieri}) and a computer,
 \begin{equation}\label{atc3}
\begin{array}{rcl}
P_1(F(X)_{\F_3},T)&=&  1+31T^5+3^5T^{10}\\
P_1(F(X)_{\F_5},T)&=& 1-57T^5+5^5T^{10}\\
P_1(F(X)_{\F_7},T)&=& 1 +7^5T^{10}\\
P_1(F(X)_{\F_{13}},T)&=& 1 +13^5T^{10}
\end{array}\qquad
\begin{array}{rcl}
D_3(F(X)_{\F_3})&=& 5^5\cdot11^2/3^{10} \\
D_5(F(X)_{\F_5})&=&11^2\cdot29^4/5^5 \\
D_7(F(X)_{\F_7})&= &2^4 \cdot5^5 \\
D_{13}(F(X)_{\F_{13}})&=&2^4\cdot 5^5  .
\end{array}
\end{equation} 
Note that $A(F(X))$ is ordinary in the first two cases and supersingular with maximal Picard number in the other two cases.\ One can easily compute the Picard numbers and write down the corresponding zeta functions if desired.\  We compute the geometric Picard numbers by a different method.\ Note that $-11$ is   a square modulo $3$ or $5$, but not modulo $7$ or $13$.

\begin{prop}
Let    $X\subset \P^4_\Z$ be the Klein cubic threefold with equation \eqref{equ3}  and let $F(X)$ be its surface of lines.\ 
Suppose $p\ne 2$.\ If $-11$ is   a square modulo $p$, the reduction modulo $p$ of $F(X)$ has geometric
Picard number   $25 $, otherwise it has geometric
Picard number $45$.\end{prop}

\begin{proof}
Set   $\nu:=\frac{-1+\sqrt{-11}}{2}$ and $E'_\C:=\C/\Z[\nu]$.
By \cite[Corollary 4, p.~138]{Adler},  $A(F(X))_\C$ is isomorphic to $(E'_{\C})^{5}$.\  
 By \cite[Appendix A3]{Silverman},
the elliptic curve $E'_\C$ has a   model  defined by the equations
\[
 y^{2}+y=x^{3}-x^{2}-7x+10 =0 
\]
over $\Q$, which we   denote by $E'$.\ 
Since $A(F(X))_\C$ and  $E_\C^{\prime 5}$ are isomorphic,  
$A(F(X)) $ and  $E^{\prime 5}$ are isomorphic over some number field (\cite[Appendix I,  p.~240]{MumfordAV}).

We use  Deuring's criterion \cite[Chapter 13, Theorem 12)]{Lang}: for odd $p\ne 11$, the reduction of~$E'$ modulo $p$ is supersingular if and only if $p$ is inert or ramified in $\Z[\nu]$.\ By classical results in number theory, an odd prime $p\ne 11$ is inert or ramified in $\Z[\nu]$ if and only if $-11$ is not a square modulo $p$.
The geometric Picard number of the reduction modulo $p$ of $A(F(X))$ is therefore $45$ if~$-11$ is not a square modulo $p$, and $25$ otherwise.
\end{proof}

\subsubsection{An implementation of our algorithm}\label{algos}
We use the notation of Section \ref{secbsd}.
Let $X\subset \P^{4}_\Z$ be the cubic threefold defined by the equation
\[
f+2q_{1}x_{4}+2q_{2}x_{5}+x_{1}x_{4}^{2}+2x_{2}x_{4}x_{5}+x_{3}x_{5}^{2}=0,
\]
 where
\begin{eqnarray*}
f&=&x_{2}^{2}x_{3}-(x_{1}^{3}+4x_{1}x_{3}^{2}+2x_{3}^{3}),\\
q_{1}&=&x_{1}^{2}+2x_{2}^{2}+x_{2}x_{3}+x_{3}^{2},\\
q_{2}&=&x_{1}x_{2}+4x_{2}x_{3}+x_{3}^{2}.
\end{eqnarray*}
It contains the line $L$   defined by the equations $ x_{1}=x_{2}=x_{3}=0 $.

In characteristics $\le 31$,
the cubic $X$ is smooth except in characteristics $2$ or $3$ and the plane quintic curve
$ \Gamma_{L}$ is smooth except in characteristics $2$ or $5$.

We   implemented in {\tt Sage} the  algorithm described in Algorithm \ref{algo} (see \cite{magma}).\
Over $\F_{5}$, we get
\[
P_{1}(F(X)_{\F_5},T)=(1+5T^{2})(1+2T^{2}+8T^{3}-6T^{4}+40T^{5}+50T^{6}+625T^{8}).
\]
It follows that $A(F(X)_{\F_5})$ is not ordinary and not simple (it contains an elliptic curve).
  We also compute    (see Remark \ref{atcsrem})
\begin{equation}\label{atc4}
D_{5}(F(X)_{\F_5})=\lim_{s\to1}\frac{P_{2}(F(X)_{\F_5},5^{-s})}{(1-5^{1-s})^{5}}=\frac{2^{18}\cdot3^{5}\cdot157}{5^{10}}.
\end{equation} 

 \medskip

 Over the field $\F_{7}$, we compute that  $P_1(F(X)_{\F_7},T) $ is equal to
$$
1+4T+15\,T^{2}+46\,T^{3}+159\,T^{4}+460\,T^{5}+1\,113\,T^{6}
+2\,254\, T^{7}+5\,145\,T^{8}+9\,604\,T^{9}+16\,807\,T^{10}.
$$
This polynomial is irreducible   over $\Q$; it follows that $A(F(X)_{\F_7})$ is   ordinary and   simple (Section~\ref{sso}).\   We can even get more by using a nice criterion from \cite{HZ}.

 \begin{prop}
 The abelian variety $A(F(X)_{\F_7})$ is   absolutely simple,   \ie, it remains simple over any field extension. 
 \end{prop}

  \begin{proof}
 We want to apply the criterion \cite[Proposition 3 (1)]{HZ} to the \av\ $A:=A(F(X)_{\F_7})$.\ Let $d>1$.\ Since the \cha\ polynomial $ Q_1(A,T)$ (which is also the minimal polynomial) of the Frobenius morphism $F$ is not in $\Z[T^d]$, it is enough to check that, for any $d>1$,  there are no $d$th roots of unity $\zeta$ such that $\Q(F^d)\subsetneq \Q(F)$ and $\Q(F^d,\zeta)=\Q(F)$.\ If this is the case, $\Q(\zeta)$ is contained in  $\Q(F)$, hence $\phi(d)$ (where 
  $\phi$ is the Euler totient function) divides $\deg\bigl(Q_1(A,T)\bigr)=10$.\ This implies $d\in\{2,3,4,6,11,22\}$.\ But for these values of $d$, one computes that the \cha\ polynomial $Q_1(A_{\F_{7^d}},T)$ of $F^d$ is irreducible (of degree 10), and this contradicts  $\Q(F^d)\subsetneq \Q(F)$.\ Thus $A$ is absolutely simple. 
  \end{proof}
 
 We also compute (see Remark \ref{atcsrem})
\begin{equation}\label{atc5}
D_7(F(X)_{\F_7})=\lim_{s\to1}\frac{P_{2}(F(X)_{\F_7},7^{-s})}{(1-7^{1-s})^{5}}=\frac{2^{4}\cdot83^{2}\cdot557\cdot5\,737}{7^{10}}.
\end{equation} 

Here are some more computations in ``high'' \cha s:
\begin{eqnarray*}
\scriptstyle{P_{1}(F(X)_{\F_{23}},T)} &\scriptstyle{=}&\scriptstyle{ 1 +21\,T^2-35\,T^3+759\,T^4-890\,T^5+17\,457\,T^6-18\,515\,T^7+255\,507\,T^8+6\,436\,343\,T^{10},}\\
\scriptstyle{P_{1}(F(X)_{\F_{29}},T)}&\scriptstyle{=}&\scriptstyle{1+3\,T+5\,T^2+15\,T^3+352\,T^4+2\,828\,T^5+10\,208\,T^6+12\,615\,T^7+121\,945\,T^8+2\,121\,843\,T^9+20\,511\,149\,T^{10},}\\
\scriptstyle{P_{1}(F(X)_{\F_{31}},T) }&\scriptstyle{=}&\scriptstyle{1+2\,T+2\,T^2+72\,T^3+117\,T^4-812\,T^5+3\,627\,T^6+69\,192\,T^7+59\,582\,T^8+1\,847\,042\,T^9+28\,629\,151\,T^{10},}
\end{eqnarray*}
 and the corresponding limits
\begin{equation}\label{atc44}
\begin{array}{rcl}
\scriptstyle{D_{23}(F(X)_{\F_{23}})} &\scriptstyle{=}&\scriptstyle{\frac{3^2\cdot 5^2\cdot 13\cdot 157\cdot 19\,861 \cdot10\,004\,497}{23^{10}},}\\
\scriptstyle{D_{29}(F(X)_{\F_{29}})}&\scriptstyle{=}&\scriptstyle{\frac{2^2\cdot 17\cdot 173\,194\,374\,702\,432\,997}{29^{10}},}\\
\scriptstyle{D_{31}(F(X)_{\F_{31}})}&\scriptstyle{=}&\scriptstyle{\frac{2^9\cdot 47^2\cdot 683\cdot 8\,087 \cdot5\,312\,689}{31^{10}}.}
\end{array}
\end{equation}

\subsubsection{Smooth cubic threefolds over $\F_2$, $\F_3$, $\F_4$, or $\F_5$ with no lines}\label{noli}

Using a computer, it is easy to find many smooth cubic threefolds defined over $\F_2$ with no $\F_2$-lines (see Example \ref{cex}).\ For example, the cubic threefold   $X\subset\P^4_{\F_2}$ defined by the equation
$$
x_1^3 + x_2^3 + x_3^3 + x_1^2 x_2+ x_2^2 x_3
+ x_3^2 x_1 + x_1 x_2x_3+x_1x_4^2 + x_1^2 x_4+ x_2 x_5^2 + x_2^2 x_5+x_4^2x_5 =0
$$
contains no $\F_2$-lines.\ We also have\footnote{Among smooth cubics in $\P^4_{\F_2}$ with no $\F_2$-lines, the computer found examples whose number of   $\F_2$-points is any odd number between $3$ and $13$.} 
$$N_1(X)=9, N_2(X)=81,  N_3(X)=657, N_4(X)=4\,225, N_5(X)=34\,049,
$$
hence (see \eqref{mr} for the definition of $M_r(X)$)
$$M_1(X)=-3, M_2(X)=-1, M_3(X)=9, M_4(X)=-9, M_5(X)=7.
$$
The polynomial $P_1(F(X),T)=P_3(X,T/2)=\prod_{j=1}^{10}(1-\omega_jT)$ is then given by
$$ \exp \Bigl(\sum_{r=1}^5M_r(X)\frac{T^r}{r}\Bigr)+O(T^6)=1-3T+4T^2-10T^4+20T^5+O(T^6).$$
Using the reciprocity property \eqref{rec}, we obtain
$$P_1(F(X),T)=1-3T+4T^2-10T^4+20T^5-10\cdot 2 T^6+4\cdot 2^3  T^8-3\cdot 2^4  T^9+2^5 T^{10}.$$
  Since this polynomial has no multiple roots, the Picard number of $F(X)$ is 5 (Corollary \ref{pic}).

We found by random computer search the smooth cubic threefold   $X'\subset \P^4_{\F_3}$ defined by the equation
$$2x_1^3 + 2x_2^3 + x_1x_3^2 + x_2^2x_4+ 2x_3^2x_4 +x_1^2x_5 + x_2x_3x_5 + 2x_1x_4x_5 + 2x_2x_4x_5 +2x_4^2x_5 + 2x_4x_5^2 + x_5^3=0.$$
It contains no $\F_3$-lines and  25 $\F_3$-points.\ Computing directly the number of points on extensions of $\F_3$, as we did above for $\F_2$, takes too much time, and it is quicker to use  the Bombieri--Swinnerton-Dyer method (Proposition \ref{proposition number of points Bombieri}) on $X'_{\F_9}$, which contains an  $\F_9$-line.\ The result is that $P_1(F(X')_{\F_9},T) $ is equal to
$$1-5T+8T^2+10T^3-124T^4+515T^5-1\,116T^6+810T^7+5\,832T^8-32\,805T^9+59\,049T^{10}.
$$
Using the fact that  $X'$ has $25$ $\F_3$-points and that the roots of $P_1(F(X')_{\F_3},T)$ are square roots of the roots of $P_1(F(X')_{\F_9},T)$, one finds 
$$P_1(F(X')_{\F_3},T) =1-5T+10T^2-2T^3-36T^4+95T^5-108T^6-18T^7+270T^8-405T^9+243T^{10},$$
and the numbers of $\F_{3^r}$-lines in $X'_{\F_{3^r}}$, for $r\in\{1,\dots,5\}$, are  $0$, $40$, $1\,455$, $5\,740$,  $72\,800$, respectively.

Similarly, the smooth cubic threefold in $\P^4_{\F_4}$ defined by the equation
$$
x_1^3 + x_1^2x_2 + x_2^3 + x_1^2x_3 + ux_1x_3^2 + ux_2x_3^2 + 
    u^2x_1x_2x_4 + x_2^2x_4 + ux_4^3 + x_2^2x_5 + ux_2x_3x_5 + 
    x_3^2x_5 + x_3x_5^2 + x_5^3=0,
$$
where $u^2+u+1 = 0$, contains no $\F_4$-lines and  61 $\F_4$-points.

Finally, the smooth cubic threefold in $\P^4_{\F_5}$ defined by the equation
\begin{multline*}
x_1^3 + 2x_2^3 + x_2^2x_3 + 3x_1x_3^2 + x_1^2x_4 + x_1x_2x_4 + x_1x_3x_4+ 
    3x_2x_3x_4 + 4x_3^2x_4 + x_2x_4^2  \\
{}+ 4x_3x_4^2 + 3x_2^2x_5 + x_1x_3x_5 + 
    3x_2x_3x_5 + 3x_1x_4x_5 + 3x_4^2x_5 + x_2x_5^2 + 3x_5^3=0
\end{multline*}
contains no $\F_5$-lines and  126 $\F_5$-points.  

We were unable to find   smooth cubic threefolds defined over $\F_q$ with no $\F_q$-lines for the remaining values $q\in\{7,8,9\}$ (by Theorem \ref{prop6},   there are always $\F_q$-lines for $q\ge 11$).

\subsubsection{Nodal cubic threefolds over $\F_2$ or $\F_3$  with no lines}
\label{nodal}

Regarding cubic threefolds with one node and no lines, we found the following examples.

The unique singular point of the 
cubic in $\P^4_{\F_2}$ defined by the equation
\[
 x_2^3 + x_2^2x_3 + x_3^3 + x_1x_2x_4 + x_3^2x_4 
 + x_4^3 + x_1^2x_5 + x_1x_3x_5 + x_2x_4x_5=0
\]
is an ordinary double point at $ x:=  (0,1,0,0,1)$ and this cubic contains no $\F_2$-lines.\ As we saw during the proof of Corollary \ref{coronodal}, the base of the cone $\T_{X,x}\cap X$ is a smooth genus-4 curve defined over $\F_2$ with no $\F_4$-points.\ The   pencils $g^1_3$ and $h^1_3$  are defined over $\F_2$.

The unique singular point of the 
cubic in $\P^4_{\F_3}$ defined by the equation
\begin{multline*}
\qquad 2x_1^3 + 2x_1^2x_2 + x_1x_2^2 + 2x_2x_3^2 + 2x_1x_2x_4 + x_2x_3x_4  \\
{}+x_1x_4^2 + 2x_4^3 + x_2x_3x_5 + 2x_3^2x_5 + x_2x_5^2 + x_5^3=0\qquad
\end{multline*}
is an ordinary double point at $x:=(1,0,0,0,1)$ and this cubic contains no $\F_3$-lines.\ Again, the base of the cone $\T_{X,x}\cap X$ is a smooth genus-4 curve defined over $\F_3$ with no $\F_9$-points, and the   pencils $g^1_3$ and $h^1_3$  are defined over $\F_3$.

\subsection{Average number of lines} \label{remaverage}
Consider the   Grassmannian $G:= \Gr(1, \P^{n+1}_{\F_q})$, the parameter space  $\P=\P\bigl(H^0\bigl(\P^{n+1}_{\F_q},\cO_{\P^{n+1}_{\F_q}}(d)\bigr)\bigr)$ for all degree-$d$ hypersurfaces in $\P^{n+1}_{\F_q} $, and the incidence variety $I=\{(L,X)\in G\times\P\mid L\subset X\}$.\ The first projection $I\to G$ is a projective bundle, hence it is easy to compute the number of $\F_q$-points of $I$.\ The fibers of the second projection $I\to \P$ are the varieties of lines.\ The average number of lines (on {\em all} degree-$d$ $n$-folds) is therefore
\begin{equation}\label{eq}
\frac{\Card\bigl(G(\F_q)\bigr) (q^{\dim(\P)-d}-1)}{q^{\dim(\P)+1}-1}\sim  \Card\bigl(G(\F_q)\bigr)  q^{\dim(\P)-d-1}
.\end{equation} 
Recall that $\Card\bigl(G(\F_q)\bigr)=\sum_{0\le i<j\le n+1}q^{i+j-1}$.\ For cubic 3-folds, the right side of \eqref{eq} is
$$q^2+q+2+2q^{-1}+2q^{-2}+ q^{-3}+ q^{-4}.$$
 For $q=2$, the average number of lines on a cubic threefold is therefore $\sim 9.688$ (compare with Example \ref{cex} below).

 \begin{exam}[Computer experiments]\upshape\label{cex}
For a random sample of  $5\cdot 10^4$
     cubic threefolds defined over $\F_2$, we computed for each the number of $\F_2$-lines.
    \begin{center}
\includegraphics[scale=.4]{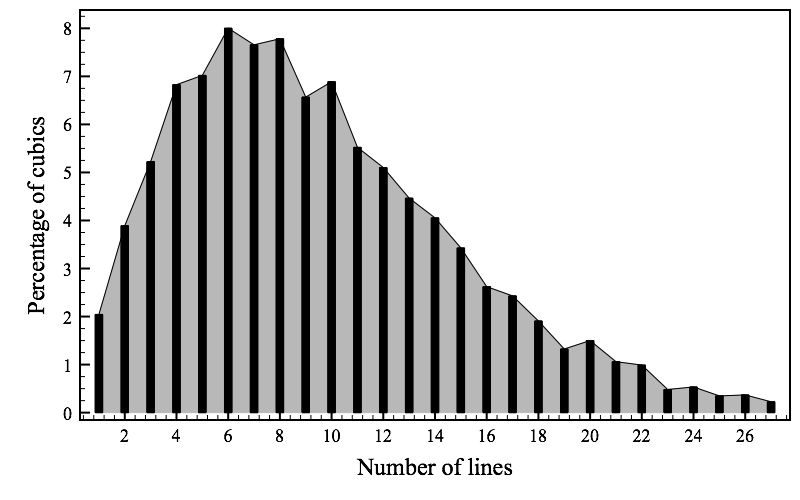}
\end{center}
The average number of lines in this sample is $\sim 9.651$.

Smooth cubic threefolds contain less lines: here is the distribution of the numbers of $\F_2$-lines for a random sample
 of $5\cdot 10^4$
  {\em smooth} cubic threfolds defined over $\F_2$. 
 \begin{center}
\includegraphics[scale=.4]{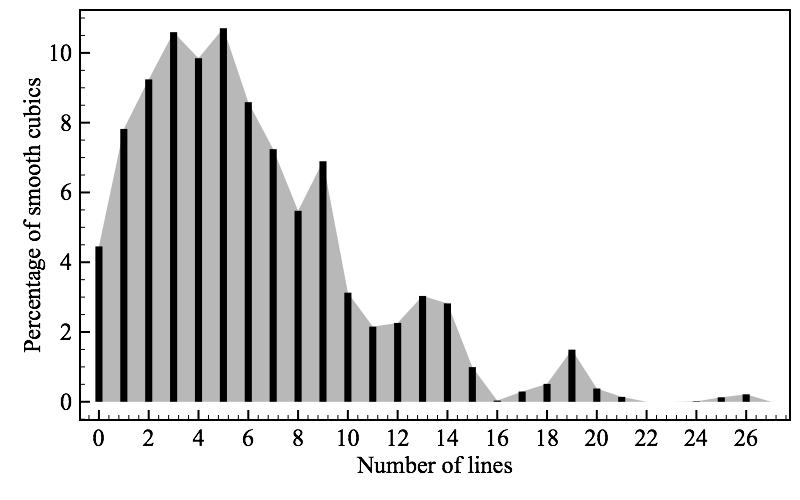}
\end{center}
The average number of lines  in this sample  is $\sim 6.963$.
\end{exam}

\section{Cubic fourfolds}

We now examine cubic fourfolds over $\F_q$.\ We expect them to contain ``more'' lines than cubic threefolds (indeed, all the examples we computed do contain $\F_q$-lines).\ Unfortunately, we cannot just take $\F_q$-hyperplane sections  and apply our results from Section \ref{secuth}, because these results
only concern  mildly singular cubic threefolds, and   there is no {\em a priori} reason why there would exist  a   hyperplane section defined over $\F_q$ with these suitable singularities.

We follow the same path as in Section \ref{secuth}.\ Recall that for any field $\kk$, the   scheme $F(X)$ of  lines contained in   a    cubic fourfold 
  $X\subset \P^5_\kk$ with finite singular set  is a geometrically connected   local complete intersection fourfold (Section \ref{segs}) with trivial canonical sheaf (\cite[Proposition (1.8)]{ak}).

\subsection{The zeta function of the fourfold of lines}

Let    $X\subset \P^5_{\F_q}$ be a {\em smooth} cubic hypersurface defined over $\F_q$.\ Its Betti numbers  are $1$, $0$, $1$, $0$, $23$, $0$, $1$, $0$, $1$, and the eigenvalues of the Frobenius morphism acting on $H^4(\overline X,\Q_\ell)$  
 are all divisible by $q$ as algebraic integers (\cite[Remark~5.1]{kat}).\ We  write
\begin{equation}\label{n1x}
N_r(X)= 1+q^r+q^{3r}+q^{4r}+q^r\sum_{j=1}^{23}\omega_j^r,
\end{equation}
where  the complex algebraic integers $\omega_j$ (and all their  conjugates) have modulus 
$q$, with $\omega_{23}=q$ (it corresponds to  the part of the cohomology that comes from $H^4( \P^5_{\overline{\F_q}},\Q_\ell)$).\ The trace formula~\eqref{zeta} reads
$$
Z(X,T)=\frac{1}{(1-T)(1-qT)(1-q^2T)(1-q^3T)(1-q^4T)P^0_4(X,T)},
$$
where 
\begin{equation}\label{defp0}
P^0_4(X,T) :=\frac{P_4(X,T)}{1-q^2T}=\prod_{j=1}^{22}(1-q\omega_jT)
.\end{equation}
 If we set 
\begin{equation}\label{mr4}
M_r(X):=
\frac1{q^r}\bigl(
N_r(X)-(1+q^r+q^{2r}+q^{3r}+q^{4r})\bigr) 
=\sum_{j=1}^{22}\omega_j^r,
\end{equation}
we obtain
\begin{equation}\label{mr44}
P^0_4(X,T)=\exp \Bigl (\sum_{r\ge 1}M_r(X)\frac{(qT)^r}{r}\Bigr)
.
\end{equation}

 \begin{theo}\label{main4}
Let    $X\subset \P^5_{\F_q}$ be a {\em smooth} cubic hypersurface defined over $\F_q$ and let $F(X)$ be the smooth fourfold of lines contained in $X$.\ With the notation above, we have $P_i(F(X),T)=0$ for $i$ odd and
\begin{eqnarray*}
P_2(F(X),T)&=& P_6(F(X),T/q^2)= P_4(X,T/q )=\prod_{1\le j\le 23}(1-\omega_jT)\\
P_4(F(X),T)&=&  \prod_{1\le j\le k\le 23}(1-\omega_j\omega_kT),
\end{eqnarray*}
where the complex numbers $\omega_1,\dots,\omega_{22}$ have  modulus $q$ and $\omega_{23}=q$, and 
\begin{equation}\label{zetafx2}
 Z(F(X),T) 
 = 
 \frac{1}
 {(1-T)(1-q^4T)\prod_{1\le j\le 23}\bigl((1-\omega_jT)(1-q^2\omega_jT)\bigr)\prod_{1\le j\le k\le 23}(1-\omega_j\omega_kT)}
.\end{equation}
\end{theo}

\begin{proof} The various methods of proof described in the proof of Theorem \ref{main} are still valid here.\ For example, one may deduce the theorem from the  isomorphisms
\begin{equation*}
H^4(\overline X,\Q_\ell)\isomto H^2\bigl(\overline{F(X)},\Q_\ell(1)\bigr) \quad{\rm and}\quad  \Sym^2 H^2(\overline{F(X)},\Q_\ell) \isomto H^4(\overline{F(X)},\Q_\ell)
\end{equation*}
obtained from the Galkin--Shinder relation \eqref{gsfgrot} (\cite[Example 6.4]{gash}) or the analogous (known) statements in \cha\ $0$.\  We leave the details to the reader.\end{proof}

 \subsection{Existence of lines over large finite fields}

As we did for cubic threefolds, we use the Deligne--Weil estimates to 
 find a lower bound for the number of $\F_q$-lines on a smooth cubic fourfold defined over $\F_q$.

\begin{theo}\label{th52}
Let $X$ be a smooth cubic fourfold defined over $\F_q$ and let $N_1 ( F(X) )$ be the   number of $\F_q$-lines contained in $X$.\ For $q\ge 25$, we have
$$N_1 ( F(X) )\ge
q^4-23q^3+276q^2-23q+1  
$$
and, for smaller values of $q$,
\renewcommand{\arraystretch}{1.5}
$$\begin{tabular}{|c|cccccccccc|}
  \hline
\strut  q &\ $5$ & \ $7$&\ $8$&\ $9$&\ $11$&\ $13$&\ $16$&\ $17$&\ $19$&\ $23$ 
\\
  \hline
 $N_1 ( F(X) )\ge$ &\  $44$\  &\  $715$\ &\ $1\,448$\ &\ $2\,572$\ &\ $6\,443$\ 
 &\ $13\,365$\ &\ $32\,245$\ &\ $41\,543$\ &\ $65\,912$\ &\ $144\,548$
 \\ 
  \hline
\end{tabular}
$$\renewcommand{\arraystretch}{1}
In particular, $X$ always contains an $\F_q$-line when $q\ge5$.
\end{theo}

When $q$ is a power of $2$, we will see in  Proposition~\ref{ked}  that $X$ always contains an $\F_q$-line.\ These leaves only the case $q=3$  open (see Section \ref{f34}).

\begin{proof}
Write the roots of  $Q_2(F(X),T)$  
as $q$ (with multiplicity $a$), $-q$ (with multiplicity $b$), $\omega_1,\dots ,\omega_c ,\overline\omega_1,\dots ,\overline\omega_c$, with $a+b+2c=23$.\ The $r_j:=\omega_j+\overline \omega_j$ are then real numbers in $[-2q,2q]$ and, by \eqref{grolef2} and Theorem \ref{main4}, we  have
\begin{eqnarray}
N_1\bigl( F(X)\bigr)&=&1 + q^4 +\sum_{1\le j\le k\le 23}\omega_j\omega_k  + (1 + q^2)\sum_{1\le j \le 23}\omega_j\label{n1fx}
\\
&=&1 + q^4 +\Bigl( \frac12\bigl( a(a+1)+b(b+1)\bigr) -ab\Bigr)q^2
+(a-b)q\sum_{1\le j\le c}r_j\nonumber\\
&&\qquad{}-cq^2+\sum_{1\le j\le k\le c}r_j  r_k+ (1 + q^2)\Bigl((a-b)q+\sum_{1\le j\le c}r_j\Bigr)\nonumber\\
&=&1 + q^4 +\frac12 \bigl( (a-b)^2+23-4c\bigr) q^2+(1 + q^2)(a-b)q\nonumber\\
&&\qquad{}+ \sum_{1\le j\le k\le c}r_j  r_k+ \bigl(1 + q^2+(a-b)q\bigr)\sum_{1\le j\le c}r_j.\nonumber
\end{eqnarray}
Since $a+b=23-2c$ is odd, it is enough to study the cases $a=1$ and $b=0$, or $a=0$ and $b=1$, since we can always consider pairs $q,q$, or $-q,-q$, as $\omega, \overline\omega$.\ We then have $c=11$ and  setting $\eps:=a-b\in\{-1,1\}$, we obtain 
\begin{equation*}
N_1\bigl( F(X)\bigr)=   q^4+\eps q^3 -10q^2+\eps q+1 +4q^2 f_u\Bigl(\frac{r_1}{2q} ,\dots,\frac{r_{11}}{2q} \Bigr),
\end{equation*}
where $u= \frac12(q+\eps +\frac{1}{q})$ and
$$\forall (t_1,\dots,t_c)\in [-1,1]^c\qquad f_u(t_1,\dots,t_c):=\sum_{1\le j\le k\le c}t_j  t_k+u \sum_{1\le j\le c}t_j
.$$
The strictly convex function $f_u$   achieves its minimum over $[-1,1]^c$ at a unique point and since $f_u$  is symmetric in the variables $(t_1,\dots,t_c)$, this point must be of the form $(t_0,\dots,t_0)$. Since $f_u(t,\dots,t)=\frac{c(c+1)}{2}t^2+cut$, we obtain (since $u>0$):
\begin{itemize}
\item either $u< c+1$ and $t_0=- \frac{u}{c+1}$, and the minimum is $-\frac{cu^2}{2(c+1)}$;
\item or $u\ge c+1$  and $t_0=- 1$,  and the minimum is $\frac{c(c+1)}{2}-cu$.
\end{itemize}
Assume first $\eps=-1$. When $q\ge 25$, we have
\begin{eqnarray*}
N_1\bigl( F(X)\bigr)&\ge &q^4- q^3 -10q^2- q+1 +4q^2\Bigl( \frac{11\cdot 12}{2}-\frac{11}{2}\Bigl(q-1 +\frac{1}{q}\Bigr)\Bigr)\\
&=&q^4- 23q^3 +276q^2- 23q+1.
\end{eqnarray*}
When $q\le 23$, we have instead
\begin{eqnarray*}
N_1\bigl( F(X)\bigr)&\ge &q^4- q^3 -10q^2- q+1 +4q^2\Bigl( -\frac{11(q-1 +\frac{1}{q})^2}{4\cdot 2\cdot 12}\Bigr)\\
&=&\frac{13}{24}q^4-  \frac{1}{12}q^3 - \frac{91}{8}q^2- \frac{1}{12}q+\frac{13}{24}.
\end{eqnarray*}
This lower bound gives the numbers in the table.
When $\eps=1$, the bounds we get are higher.
%
\end{proof}

\subsection{Existence of lines over some finite fields}

The cohomology of the structure sheaf of the fourfold $F(X)$ is particularly simple and this can be used to prove congruences for its number of  
$\F_q$-points by using the Katz formula \eqref{trka}.

\begin{prop}[Altman--Kleiman]\label{ak}
Let $X\subset \P^5_\kk$ be a cubic hypersurface defined over a field $\kk$, with finite singular set.\ We have
\begin{eqnarray*}
&h^0(F(X),\cO_{F(X)})=h^2(F(X),\cO_{F(X)})=h^4(F(X),\cO_{F(X)})=1&\\
&h^1(F(X),\cO_{F(X)})=h^3(F(X),\cO_{F(X)})=0.&
\end{eqnarray*}
\end{prop}

\begin{proof}
The scheme $F(X)$ is the zero scheme of a section    of the rank-4 vector bundle $\cE^\vee:= \Sym^3\!\cS^\vee$ on $G:=\Gr(1,\P^5_\kk)$ and
the  Koszul complex
 \begin{equation}\label{kr}
0\to\bw{4}{\cE}\to\bw{3}{\cE}\to \bw{2}{\cE} \to \cE 
\to
 \cO_G\to \cO_{F(X)}\to 0
\end{equation}
is exact.\ By \cite[Theorem (5.1)]{ak}, the only nonzero cohomology groups of $\bw{r}{\cE} $ are
$$H^8(G,\bw{4}{\cE})\isom H^4(G,\bw{2}{\cE})\isom \kk.
$$
Chasing through the cohomology sequences associated with \eqref{kr}, we obtain $H^1(F(X),\cO_{F(X)})=H^3(F(X),\cO_{F(X)})=0$ and
\begin{eqnarray*}
H^0(F(X),\cO_{F(X)})&\isom&H^0(G,\cO_G),\\
H^2(F(X),\cO_{F(X)})&\isom&H^4(G,\bw{2}{\cE}),\\
H^4(F(X),\cO_{F(X)})&\isom&H^8(G,\bw{4}{\cE}).
\end{eqnarray*}
This proves the proposition. \end{proof}

Since $\omega_{F(X)}$ is trivial, the multiplication product
\begin{equation}\label{cup}
H^2(F(X),\cO_{F(X)})\otimes H^2(F(X),\cO_{F(X)})\to H^4(F(X),\cO_{F(X)})
\end{equation}
is the Serre duality pairing.\ It is therefore an isomorphism.

 \begin{coro}\label{coro}
Let $X\subset \P^5_{\F_q}$ be a cubic hypersurface with finite singular set, defined over  $\F_q$.\ If   $q \equiv 2\pmod3$, the hypersurface $X$ contains  an $\F_q$-line.
\end{coro}

\begin{proof}
The   $\F_q$-linear map  $\gF_q$ defined in Section \ref{skatz} acts on the one-dimensional $\F_q$-vector space $H^2(F(X),\cO_{F(X)})$ (Proposition \ref{ak}) by multiplication by some $\lambda\in \F_q$; since \eqref{cup} is an isomorphism, $\gF_q$ acts on $H^4(F(X),\cO_{F(X)})$   by multiplication by $\lambda^2$.\ 
It then follows from the Katz formula~\eqref{trka} that we have
$$N_1(F(X))\cdot 1_{\F_q}= 1+\lambda+\lambda^2\qquad\hbox{in }\F_q .
$$
If $ 1+\lambda+\lambda^2= 0_{\F_q}$, we have $\lambda^3=1_{\F_q}$.
 Since $3 \nmid q-1$, there are no elements of order $3$ in $\F_q^\times$, hence the morphism $\F_q^\times\to \F_q^\times$, $x\mapsto x^3$ is injective.\ Therefore, $\lambda  = 1_{\F_q}$, hence $ 3\cdot 1_{\F_q}=  1_{\F_q} $, but this contradicts our hypothesis.

We thus have   $ 1+\lambda+\lambda^2\ne 0_{\F_q}$, hence  $N_1(F(X))  $ is not divisible by the \cha\ of $\F_q$ and the corollary is proved. \end{proof}

 The  proof of the next statement was provided to us by K.~Kedlaya, who kindly allowed us to include it in this article.
 
  \begin{prop}[Kedlaya]\label{ked}
Let $X\subset \P^5_{\F_q}$ be a smooth cubic hypersurface  defined over  $\F_q$, where $q$ is a power of $2$.\ One has $N_1(F(X)) \equiv 1\pmod2$; in particular,   $X$ contains  an $\F_q$-line.
\end{prop}

\begin{proof}
With the notation of \eqref{n1x}, we may write
$$
N_1(X)\equiv 1+q+q\sum_{j=1}^{23}\omega_j \pmod{q^2},
$$
where we interpret this equation   in the integral closure of $\Z_2$ in some
algebraic closure of $\Q_2$. In particular, the residue  of $\sum_{j=1}^{23}
\omega_j$ is an element of $\F_2$, which we call $\lambda$.

Since $h^{0,2}(F(X))=1$, by  Mazur's theorem (\cite[First form of the conjecture, p.~656]{maz}), at most one of the roots of the polynomial $Q_2(F(X), T)$ (which, by Theorem~\ref{main4}, are the $\omega_j$) is a $2$-adic unit (the others have residue $0$). Consequently, among the $\omega_j$, there is  one
with residue $\lambda$ and the others have residue 0. By \eqref{n1fx}, we
obtain
$$N_1(F(X)) \equiv 1 + \lambda + \lambda^2 \equiv 1 \pmod{2}.$$
This proves the proposition.
 \end{proof}
  
 \subsection{Examples of cubic fourfolds}
 
 \subsubsection{Fermat cubics}
If  $X\subset \P^5_{\F_p}$ is the   Fermat fourfold, it is a simple exercise to write down the zeta function 
of $F(X)$ using Theorem \ref{thweil} and    Theorem \ref{main4}, as we did in dimension 3 in Corollary \ref{coroweil}.

\subsubsection{Cubic fourfolds over $\F_2$ with only one line}\label{sec542}

Smooth cubic fourfolds defined over $\F_2$ always contain an  $\F_2$-line by Corollary \ref{coro}.\ Random computer searches produce
  examples with exactly one  $\F_2$-line: for example, the only $\F_2$-line contained in the smooth cubic fourfold defined by the equation
\begin{multline*}
x_1^3 + x_2^3 + x_3^3 + x_1^2 x_2+ x_2^2 x_3+x_2x_4x_5
+ x_3^2 x_1 + x_1 x_2x_3+
 x_1 x_4^2 + x_1^2 x_4+ x_2 x_5^2 + x_2^2 x_5\\
 {} + x_4^2x_5+ x_4x_5^2+ x_3 x_6^2 + x_3^2 x_6+ x_4^2x_6+ x_4x_6^2+ x_5^2x_6+ x_5x_6^2+x_4x_5x_6=0
\end{multline*}
  is the line $\langle  (0 , 0 , 0 , 0 , 1 , 1),  (0 , 0 , 0 , 1 , 0 , 1) \rangle$; the fourfold contains  13 $\F_2$-points.

%
  
    \subsubsection{Cubic fourfolds over $\F_3$}\label{f34}
  
  Our results say nothing about the existence of lines in smooth cubic fourfolds defined over $\F_3$.\ Our computer searches only produced fourfolds containing lines (and  both cases $N_1(F(X)) \equiv 0$ or $1\pmod3$ in the proof of Corollary~\ref{coro} do occur), leading us to suspect that all (smooth) cubic fourfolds defined over $\F_3$   should contain lines. 

\section{Cubics of dimensions 5 or more}\label{5ormore}

In higher dimensions,   the existence of lines  is easy to settle.

\begin{theo}
Any cubic hypersurface $X\subset \P^{n+1}_{\F_q}$ of dimension $n\ge 6$ defined over  $\F_q$ contains $\F_q$-points and through any such point, there is an $\F_q$-line contained in $X$.
\end{theo}

\begin{proof}
This is an immediate consequence of the Chevalley--Warning theorem: $X(\F_q)$ is nonempty because $n+2>3$ and given $x\in X(\F_q)$, lines through $x$ and contained in $X$ are parametrized by a subscheme of $\P^n_{\F_q}$ defined by equations of degrees $1$, $2$, and $3$ and coefficients in $\F_q$.\ Since $n+1>1+2+3$, this subscheme contains an $\F_q$-point. 
\end{proof}

The Chevalley--Warning theorem   implies $N_1(X)\ge \frac{q^{n-1}-1}{q-1}$.\ When $n\ge 6$, we obtain from the theorem   $N_1 (F(X) )\ge \frac{q^{n-1}-1}{q^2-1}$; when $X$ (hence also $F(X)$) is smooth, the Deligne--Weil estimates for $F(X)$ provide better bounds.

When $n\ge 5$, we may also use the fact that   the scheme of lines contained in a smooth cubic hypersurface is a Fano variety (its anticanonical bundle $\cO(4-n)$ is ample).

\begin{theo}
Assume $n\ge 5 $ and let  $X\subset \P^{n+1}_{\F_q}$ be any   cubic hypersurface defined over $\F_q$.\ The number of $\F_q$-lines contained in $X$ is $\equiv 1\pmod{q}$.
\end{theo}

\begin{proof}
When $X$ is smooth, the variety $F(X)$ is also smooth, connected, and a Fano variety.\ The result then follows from \cite[Corollary 1.3]{esn}.

To prove the result in general, we consider  as in Remark \ref{remaverage} the parameter space $\P$ for all cubic hypersurfaces in $ \P^{n+1}_{\F_q}$ and the incidence variety $I=\{(L,X)\in G\times\P\mid L\subset X\}$.\ The latter is   smooth and geometrically irreducible; the projection $\pr\colon I\to \P$ is dominant and its geometric generic fiber is a (smooth connected) Fano variety (\cite[Theorem (3.3)(ii), Proposition (1.8), Corollary (1.12), Theorem (1.16)(i)]{ak}).\ It follows from \cite[Corollary 1.2]{fak} that for any $x\in \P(\F_q)$ (corresponding to a cubic hypersurface $X\subset \P^{n+1}_{\F_q}$  defined over  $\F_q$), one has $\Card\bigl( \pr^{-1}(x)(\F_q)\bigr)\equiv 1\pmod{q}$.\ Since $\pr^{-1}(x)=F(X)$, this proves the theorem.\end{proof}

 \end{document}